\documentclass{amsart}
\usepackage{graphicx,amssymb,epsfig,amsmath}
\usepackage{lipsum}
\usepackage{enumerate}
\usepackage[hidelinks]{hyperref}
\usepackage[all]{xy}

\usepackage{relsize}
\usepackage{xcolor}
\usepackage{amscd,amssymb,hyperref}
\usepackage[all]{xy}

\def\la#1{\hbox to #1pc{\leftarrowfill}}
\def\ra#1{\hbox to #1pc{\rightarrowfill}}

\theoremstyle{plain}
\newtheorem{prop}[subsection]{Proposition}
\newtheorem{thm}[subsection]{Theorem}
\newtheorem{lem}[subsection]{Lemma}
\newtheorem{rem}[subsection]{Remark}
\newtheorem{cor}[subsection]{Corollary}
\newtheorem{remark}[subsection]{Remark}
\newtheorem{defn}[subsection]{Definition\rm}

\newcommand{\bpr}{\begin{proposition}}
\newcommand{\epr}{\end{proposition}}
\newcommand{\bco}{\begin{corollary}}
\newcommand{\eco}{\end{corollary}}
\newcommand{\ble}{\begin{lem}}
\newcommand{\ele}{\end{lem}}
\newcommand{\bde}{\begin{definition}\rm}
\newcommand{\ede}{\end{definition}\rm}
\newcommand{\bre}{\begin{remark}\rm}
\newcommand{\ere}{\end{remark}}
\newcommand{\bex}{\begin{example}\rm}
\newcommand{\eex}{\end{example}}
\numberwithin{equation}{section}

\usepackage{tikz}

\newcommand{\map}{\hbox{Map}}
\newcommand{\bmap}{\hbox{Map}^*}

\def\hofib{\operatorname{hofib}}

\newcommand{\R}{{\mathbb R}}

\newcommand{\N}{{\mathbb N}}

\newcommand{\lra}{\longrightarrow}

\def\bbr{\mathbb R}
\def\lrar{\longrightarrow}

\title{Geodesic Models for Loop Spaces on Riemannian Manifolds}

\author[A.~Bahri]{Abbas~Bahri}

\author[F. R.~Cohen]{Frederick R.~Cohen}

\author[M.~Ghazel]{Moncef~Ghazel}
  \address{Facult\'e des Sciences de Tunis\\ Universit\'e Tunis, El Manar (Tunisia)}\email{moncef.ghazel@fst.utm.tn}
  
\author[S.~Kallel]{Sadok~Kallel}
  \address{American University of Sharjah (UAE) and Laboratoire Painlev\'e (France)}\email{sadok.kallel@univ-lille.fr}


\begin{document}

\begin{abstract}
A topological group, defined in terms of ``composable minimal geodesics'' is constructed and shown to be homotopy equivalent to the pointed loop space of a path-connected, simply connected Riemannian manifold \(M\) of positive convexity radius. This construction is analogous to J.~Milnor's free group model \cite{Milnor} for the pointed loop space of a connected simplicial complex. Related models built from composable minimal geodesics are also obtained for the free loop space of \(M\) and for the space of maps from any finite two-dimensional CW-complex into \(M\).
\end{abstract}
\maketitle

\section{Introduction}

The purpose of this article is to investigate the space of loops (both based and unbased) on a Riemannian manifold\footnote{Always Hausdorff, second countable and finite dimensional.}. This is of course a very old and classical problem that dates back to  Morse, Bott and Samelson  \cite{bott, Milnor, oancea}. In his remarkable book \cite{Milnor2}, Milnor  introduces spaces of piecewise or broken geodesics as finite dimensional approximations to spaces of smooth loops with bounded energy. More recently, these piecewise geodesics started appearing as Reeb chords on cotangent bundles. In the work of the first author \cite{Abbas}, such paths are formed by pieces of the Reeb orbits and pieces of the orbits of a transverse vector field (in the kernel of the contact form), the idea being to understand how the index of the Poincar\'e map breaks after each marked point of the broken path.  
Going beyond pure mathematics, piecewise geodesic paths and loops are increasingly used in connection with nowadays `machine learning', in the attempt to model biological and engineering systems evolving over time  along the geodesics of a ``learned latent Riemannian manifold'' \cite{burge}.

Small models\footnote{A space or a model is ``small'' if it is $\sigma$-compact (i.e. a countable union of compact spaces).} for loop spaces have been of great use in both algebraic topology and geometry, with a sample of applications in \cite{Abbas, smale, ali, gk, hess, Milnor2}. 
This paper provides a number of homotopy equivalent  models, or ``approximations'', for the loop space $\Omega M$ of a Riemannian manifold $M$ of positive convexity radius $\operatorname{conv} (M)>0$; the topology of the space of loops being the compact-open topology. One set of models is to consider the {subspace} of broken geodesic loops, and its even smaller subspace of \textit{uniform} broken geodesic loops, and verify that these are indeed approximations for $\Omega M$ (Theorem \ref{pgapprox}). The other set of models are combinatorial in nature and consist of finite configurations of ordered points in $M$ (the \textit{markings}) close enough to be related by unique minimal geodesic segments,  such that identical neighbors are identified, and in the case of the group model, identical second neighbors are identified. The collection of such classes of configurations can be assembled into a  colimit of finite-dimensional models that is then shown to be homotopy equivalent to the loop space. The end result is the construction of several $H$-space models of $\Omega M$ (Theorem \ref{pgapprox}, Theorem \ref{happrox}) and a group model (Theorem \ref{reducedapprox}). The precise meanings of an \textit{H-space model} (or $H$-model) and of a \textit{group model} are given in \S\ref{models}. 

 Throughout, all spaces are assumed to lie in $kTop$, the category of compactly generated spaces (see Appendix \ref{hspace}). All loops in this paper, except in Appendices \ref{ordinary} and \ref{holonomyappendix}, are \textit{Moore loops}. 
Theorem \ref{main} summarizes the main results of the paper. Before stating it, we briefly and informally describe the spaces  appearing in the theorem, referring the reader to the body of the paper for precise definitions.
A `broken geodesic' path is made out of a chain of minimal geodesic segments. The points of the path where consecutive segments join are the `markings'.  We write $\Omega^{bg}(M)\subset \Omega M$ the subspace of broken geodesic Moore loops, with the compact-open topology, and write $\Omega^{ubg}(M)$ its subspace of uniform broken geodesics. Uniformity means unit speed across all segments of the piecewise path or loop.

\begin{thm}\label{main}
Let $(M,x_0)$ be a based simply connected Riemannian manifold of positive convexity radius $\operatorname{conv}(M)>0$. Then:
\begin{itemize}
\item $\Omega^{bg}(M)$ is an $H$-space model for $\Omega M$. 
\item $\Omega^{ubg}(M)$ is a small $H$-model for $\Omega M$. 
\item The sequences of marked points produce a small $H$-model\footnote{This is a monoid object in \(k\mathbf{Top}\); equivalently, its multiplication is strictly associative and admits a strict two-sided unit.} $H_\infty(M)$ of $\Omega M$. It has a topological group quotient $G_\infty(M)$ which is a group object in $kTop$ and a group model for $\Omega M$.
\item All models fit into the following zigzag of  homotopy equivalences and monoid maps
$$\xymatrix{
G_\infty(M)&H_{\infty}(M)\ar[r]\ar[l]&\Omega^{ubg} (M)\ar@{^(->}[r]&\Omega^{bg}(M)\ar@{^(->}[r]&\Omega M}
$$
\end{itemize}
The isomorphism types of $H_\infty(M)$ and $G_\infty(M)$, as a topological monoid and a topological group, respectively, are invariants of the based Riemannian manifold $(M,x_0)$ up to basepoint-preserving Riemannian isometry.
\end{thm}

The overall approach and inspiration of this work has been to follow the footsteps of Milnor
\cite{Milnor} whose model is modified to go from simplicial segments to geodesic segments. We give a partial exposition of those
methods, albeit in a  different and extended context.  Once the models are constructed, showing that the maps are weak-equivalences proceeds in the usual way and requires that we exhibit these spaces as fibers of suitable path-loop fibrations/quasifibrations. Perhaps the greatest hurdle has been to show that the path space $PG_\infty$ of the corresponding group model $G_\infty$ is contractible  (see \S\ref{contractibilityproof}). The weak homotopy equivalences upgrade to genuine homotopy equivalences because all spaces in sight are in hCW (Theorem \ref{pgapprox} and Theorem \ref{happrox}).

As an application of the existence of \(G_\infty\) with \(BG_\infty\simeq M\), we obtain a homotopy model for the free loop space \(LM\), and more generally a model for the mapping space \(\operatorname{Map}(X,M)\), where \(X\) is any fixed finite two-dimensional CW complex. The following result, which we state here as a corollary of our constructions, encompasses all of these cases. This is Theorem \ref{XtoBG} in the text.

\begin{cor}\label{XtoBG}
Let \(M\) be a simply connected Riemannian manifold of positive
convexity radius, and let \(G=G_\infty\) be the group model constructed
above. Let \(X\) be a finite connected CW complex
obtained as the cofiber of a map
$C\colon \bigvee_{j=1}^{m} S^1
   \longrightarrow \bigvee_{i=1}^{n} S^1, m,n\geq 0$.
Writing \(F_m\) and \(F_n\) for the fundamental groups of these wedges,
let $$R\colon
\operatorname{Hom}(F_n,G)\cong G^n
\longrightarrow
\operatorname{Hom}(F_m,G)\cong G^m
$$
be precomposition with the induced map
\(C_\#\colon F_m\to F_n\).
Then there is a homotopy equivalence
\[
\operatorname{Map}(X,M)
\simeq
PG_\infty\times_G\hofib (R)
\]
where \(G\) acts on \(\hofib (R)\) by conjugation (see \S\ref{surfaces}).
\end{cor}

For the benefit of the ``user'', we formulate in \S\ref{collection} a streamlined criterion for when a collection of loops in a based space $X$ can be a weak/homotopy approximation for $\Omega X$ (Proposition \ref{distinguished}),  and in Appendix A, we list many relevant  topological properties of the spaces of piecewise geodesic loops which can possibly be of interest to the machine learning community.

We briefly explain the genesis of this paper. The seeds of the article go back to conversations between Abbas and Fred in 2002. They proposed the group model \(G_\infty\) of \S\ref{groupmodel} and posted a preliminary version of their work on arXiv \cite{bc}. Following their work on thin loop spaces \cite{gk}, Fred Cohen invited one of the younger authors to complete and expand the paper jointly. Less than a year into this collaboration, Fred passed away. Abbas and Fred are no longer with us. The remaining authors decided to complete \cite{bc} for this dedicated special volume and to expand it into its present form. They are solely responsible for any shortcomings and feel deeply privileged to be co-signing this work with the late authors.

\vskip 5pt
{\sc Human-AI statement}: ChatGPT 5.6 was used to proofread the paper, correct minor glitches, and improve the presentation. In particular, it drew our attention to reference [27], which allowed us to  shorten significantly our argument for passing from weak equivalences to genuine homotopy equivalences in the statement of our main result.

\vskip 5pt
{\sc Acknowledgments}: The first two authors thank Ryan Budney and Mike Gage for helpful discussions. The last two authors likewise thank Ali Maalaoui. The first author is grateful for the hospitality of the Institute for Advanced Study, where part of this work was carried out. He was partially supported by National Science
Foundation and the Institute for Advanced Study (in 2006).  Sadok Kallel was supported by Grant FRG23-E-S77 AS1631 of the American University of Sharjah, UAE. 

\tableofcontents
\section{Group and H-Space Models for Loop Spaces}\label{models}

All spaces considered below are in $kTop$ (see Appendix \ref{hspace}). For any space $X$, let $X^{\mathbb{R}_{+}}$ be the space of continuous maps from $\mathbb{R}_{+}$ to $X$ with the compact-open topology.
Although most of the results we state in this paper are valid for ordinary loop spaces, we opt to work with Moore loops to avoid unnecessary technicalities. Appendix A collects useful results relevant to ordinary loops. Henceforth in this entire paper, aside from Appendix A, all paths and loops are in the sense of Moore. The basepoint is always refered to as $x_0$.

A (Moore) path in $X$  consists of a pair $(\gamma,r)$, where $r\in \mathbb{R}_{+}$ and $\gamma:[0,r] \lra X $ is a continuous map, then extended to $\bbr^+$ by setting $\gamma (t)=\gamma (r), t\geq r$. Writing a Moore path $\gamma$ as defined on $[0,r]$ means taking its restriction on that interval. The space of Moore paths in $X$, denoted by $\operatorname{Path}(X)$, is the set of all such pairs $(\gamma,r)$ 
topologized as a subspace of $X^{\mathbb{R}_{+}}\times \mathbb{R}_{+}$.

As customary, for a based space $X$, $PX\subset \operatorname{Path}(X)$ will denote the space of all \textit{based paths}; that is those starting at the basepoint $x_0\in X$. The projection $p: PX\rightarrow X$, $(\gamma,r) \rightarrow \gamma(r)$ is a fibration. The loop space $\Omega X$ of $X$ is the fiber of $p$ over $x_0$. 

Given two paths $(\gamma_1,r_1) $  and $(\gamma_2,r_2)$ in $X$ such that $\gamma_1(r_1)=\gamma_2(0)$, their concatenation, denoted by $(\gamma_1,r_1)*(\gamma_2,r_2) $, is the path $(\gamma_1 *\gamma_2, r_1+r_2) $ in $X$, where
$\gamma_1 *\gamma_2:[0,r_1+r_2]\rightarrow X$ is given by 
\begin{equation}\label{productpath}
\gamma_1*\gamma_2(t) =
\begin{cases}
\gamma_1(t),& 0\leq t\leq {r_1}\\
\gamma_2( t-r_1),& r_1\leq t\leq r_1+ r_2
\end{cases}
\end{equation}
Concatenation of paths induces a strictly associative $H$-space structure on $\Omega X$ with a strict unit. It also induces a left action of $\Omega X$ on $PX$ in the standard way.

\begin{defn}\rm
An $H$-space $G$ is an \textit{$H$-space model} (or simply $H$-model) of the loop space of a based connected space $X$ if there is a zigzag of $H$-spaces and $H$-maps, that are weak-homotopy equivalences, relating $G$ and $\Omega X$
$$\xymatrix{G&W_1\ar[l]\ar[r]&\cdots&W_k\ar[r]\ar[l]&\Omega X}
$$
In the simplest case, an \(H\)-model for \(\Omega X\) is a monoid object \(G\) in \(k\mathbf{Top}\), together with a morphism
of monoid objects  $G\rightarrow\Omega X$ that is weak equivalence. Similarly, it could be a map the other way around $\Omega X\rightarrow G$ with the same properties (contrast with \cite{neisendorfer}, Section 6.8). 

Similarly, a topological group $G$ is a \textit{group model} of $\Omega X$ if the same zigzag holds with the $W_i$'s being monoids and the maps $A_\infty$-weak equivalences. Because we choose the $W_i$'s to be strictly associative, and so is $\Omega X$, the higher homotopy coherences in the $A_\infty$ maps are only there to correct the lack of strict multiplicativity of the maps.

Note finally that ``being related by a zigzag'' is an equivalence relation.
\end{defn}

We will indicate below two basic criteria for producing loop space models. The most basic criterion comes from ``path-loop type'' fibrations, or fibrations with contractible total space and an $H$-space acting on it.  As in \cite{ganea}, we say that $F\rightarrow E \rightarrow B$ is an $h$-principal fibration if there are maps $\mu$ and $\nu$ such that the diagram
$$\xymatrix{F\times F\ar[r]^{1\times i}\ar[d]^\mu&F\times E\ar[r]^{0\times p}\ar[d]^\nu&{*}\times B\ar[d]\\
F\ar[r]^i&E\ar[r]^p&B}
$$
strictly commutes, $p$ is a fibration and $\mu$ is the $H$-space structure of $F$.
Any principal $G$-bundle in the differential topological sense is a principal fibration in this sense. The path-loop fibration $$\xymatrix{\Omega X\ar[r]& PX\ar[r]^{ev}& X}\ ,\ ev ((\gamma,r)) = \gamma (r)$$ is h-principal. The action of $\Omega X\times PX\rightarrow PX$ sends $((f,s),(\gamma,r))$ to the loop product $(f,s)*(\gamma,r)$     ($f$ applied first, then $\gamma$).

Given a Hurewicz fibration $p : E{\lrar} B$ with fiber $F=p^{-1}(x_0)$, its holonomy
action is a map
\begin{equation}\label{holonomy} h: \Omega B\times F\lrar F
\end{equation}
which to $((\gamma,r), f)$ associates the point $\tilde\gamma (r)$, where $(\tilde\gamma,r)$ is a lift of the loop $(\gamma,r)$ that starts at $f\in F$ \cite{hilton, stasheff}. By the homotopy lifting property of the fibration, there is a continuous choice of such lifts, and a different choice of lifts yields a homotopic map $h'$, so holonomy (or transport) is uniquely defined, up to homotopy, and in fact it classifies fibrations in an appropriate way \cite{stasheff}. In \cite{may}, May rectifies this holonomy action to a strict action $\Omega B\times F'\rightarrow F'$ on a homotopy equivalent model $F'$ of $F$, and obtains a more straightforward classification result in terms of an associated model for the action.

If $h$ is the holonomy of a fibration $F\rightarrow E\rightarrow B$, $e\in F$ a fixed basepoint, we call \textit{boundary map} the map
$\phi : \Omega B\rightarrow F$ which is the restriction $\phi ((\gamma,r)) = h((\gamma,r), e)$.

\begin{lem}\label{hmap} Let $E\rightarrow B$ be a fibration with based fiber $F$. If the fibration is h-principal, then the boundary map $\phi : \Omega B\rightarrow F$ is an $H$-map.
\end{lem}

\begin{proof}
    This is a direct consequence of Lemma 2.1 of \cite{ganea} and of homotopy transitivity.  Ganea shows that when the fibration is h-principal, the following homotopy commutes
$$\xymatrix{
\Omega B\times F\ar[dr]^h\ar[rr]^{\phi\times 1}&&F\times F\ar[dl]^\mu\\
&F
}
$$
On the other hand, homotopy transitivity of $h$  is the statement that the following diagram homotopy commutes \cite{hilton}
$$\xymatrix{
\Omega B\times \Omega B\times F\ar[d]^{*\times 1}\ar[r]^{\ \ 1\times h}&\Omega B\times F\ar[d]^h\\
\Omega B\times F\ar[r]^h&F}$$
where $*$ is loop multiplication.
We can use this second diagram to expand the first to the left as follows
$$\xymatrix{
\Omega B\times\Omega B\ar[r]^{1\times \phi}\ar[d]^{*}&\Omega B\times F\ar[dr]^h\ar[rr]^{\phi\times 1}&&F\times F\ar[dl]^\mu\\
\Omega B\ar[rr]^\phi&&F
}
$$
The larger commuting diagram is precisely the claim that $\phi$ is an $H$ map.
\end{proof}


\begin{cor}\label{recognition1} An $H$-space $G$ is an $H$-model for a connected loop space $\Omega B$ if there is an $h$-principal fibration $G\rightarrow E\rightarrow B$, with $B$ simply connected and 
$E$ a contractible space. If $G$ is a group, and the fibration is a principal $G$-bundle, then $G$ is a group model for $\Omega B$.
\end{cor}

\begin{proof} It suffices to observe that the boundary map for $G\rightarrow E\rightarrow B$ is an $H$-map (Lemma \ref{hmap}) and is a weak equivalence since $E$ is contractible. Assume now that $G\rightarrow E\rightarrow B$ is a principal $G$-bundle. The boundary map $\partial : \Omega B\rightarrow G$ is generally not strict, but it is $A_\infty$. The way to see this indeed is to consider the adjoint map to \eqref{holonomy} which is a map $\Theta : \Omega B\rightarrow \hbox{hAut} (F)$, where $\hbox{hAut}(F)$ is the monoid of self-homotopy equivalences of the fiber, and then identify its image with the submonoid of $G$-equivariant self-maps of the $G$-torsor $F$. This submonoid is isomorphic to $G$, and so the map $\Theta$ factors through $G$, producing precisely the map $\partial$. Since Stasheff shows explicitly that $\Theta$ is $A_\infty$ \cite{stasheff}, the map $\partial$ is also $A_\infty$. This map is also a weak equivalence if $E$ is contractible.
\end{proof}


\subsection{Families of Loops}\label{collection}

A second more geometric method to construct models for loop spaces of appropriate based spaces is to restrict attention to loops having specific geometric properties (e.g., smooth or piecewise smooth, piecewise linear, Legendrian, piecewise geodesic as in this paper, etc), and show that this subspace of loops is enough to exhaust the weak homotopy type of the loop space. This idea is classical going back to work of Morse on minimal geodesics and work of Smale on regular and Legendrian curves (eg. \cite{ali}). 
It has also been implemented for example in (\cite{gk}, Proposition 2.3) for the family of ``thin loops''. We formalize this approach below.
  
Let $X$ be a based space, and $P^{\mathcal C}X\subset PX$ a subspace of based (Moore\footnote{Of course the criterion works as well for ordinary (i.e. non Moore) loops, but we stick to our notation throughout.}) paths starting at the basepoint $x_0$. We can refer to an element of $P^{\mathcal C}X$ as a $\mathcal C$-path. The restriction of the evaluation fibration for $PX$ gives a map 
\begin{equation}\label{evaluation} ev : P^{\mathcal C}X\rightarrow X\ \ ,\ \ ev((\gamma, r))= \gamma (r)
\end{equation} We define a \textit{distinguished $\mathcal C$-cover} of $X$ to be an open cover $\{U_\alpha\}_{\alpha\in I}$ such that the restrictions $ev^{-1}(U_\alpha )\rightarrow U_\alpha$ have \textit{lifting functions} for all $\alpha\in I$. To recall, a lifting function for $U$ in the cover assigns to every Moore path $(\gamma,r)\in ev^{-1}(U)$ and to every ordinary path $\lambda$ in $U$ with $\lambda (0) = \gamma (r)$, 
a lifted ``path of paths'' in the total space starting at $(\gamma ,r)$,  depending continuously on both variables $(\gamma,r)$ and $\lambda$.

\begin{rem}\rm 
A simpler condition would have been to demand that $ev^{-1}(U)\rightarrow U$ be a fibration, since maps with the lifting criterion are (Hurewicz) fibrations. We choose, however, to specify how this fibration comes about, i.e. by continuously ``lifting paths to paths of paths'', since in applications, this is the most natural, and also most explicit way to obtain the fibration property when working with specific families of paths and loops (eg. \cite{gk}). This idea is used here in Theorem \ref{pgapprox} and Lemma \ref{lem:free loop fibration}.
\end{rem}

We first   recall that a map $\pi : E\rightarrow B$ is \textit{locally a fibration} if it is a fibration over every open set of an open cover of $B$. All our fibrations are Hurewicz fibrations. The following criterion is commonly attributed to Hurewicz and Dold.

\begin{thm}\label{Dold} (\cite{spanier}, chap.2, Theorem 13).
If $B$ is a paracompact Hausdorff space, a continuous map $\pi:E
\to\ B$ is a fibration if and only if the map $\pi$ is locally a
fibration.
\end{thm}

\begin{prop}\label{distinguished} Let $X$ be a path-connected paracompact Hausdorff based space, and $P^{\mathcal C}X$ a contractible subspace of based paths in $X$. Define $\Omega^{\mathcal C}X = ev^{-1}(x_0)$. If $X$ has a  distinguished $\mathcal C$-cover, then the inclusion  
$\Omega^{\mathcal C}X\hookrightarrow \Omega X$ is a weak homotopy equivalence. If $X$ has the homotopy type of a CW complex, then this inclusion is a homotopy equivalence. 
\end{prop}

\begin{proof}
The point is that under the hypotheses, the evaluation map $ev: P^{\mathcal C}X\rightarrow X$
is a fibration. 
To see this, let $U$ be an element of the cover, and $ev_{|U} :ev^{-1}(U)  \longrightarrow  U$ the restriction of $ev$ to $U$. Since $ev$ has a continuous lifting function on $U$, it is a Hurewicz fibration over that restriction (\cite{spanier}, chap.2). Since this is valid for all $U$ in the cover, it follows that $ev$ is a fibration by Theorem \ref{Dold}. 
The fibration \eqref{evaluation} embeds into the path loop fibration for $X$, with both base spaces being the same and both total spaces being contractible. By the long exact sequence in homotopy groups, $\Omega^{\mathcal C} X\rightarrow \Omega X$ is a weak equivalence. Finally, when $X$ is in hCW, and since $P^{\mathcal C}X$ is contractible, so also in hCW, it follows by Proposition 3 of \cite{schon} that the fiber $\Omega^{\mathcal C}(X)$ is also in hCW. Since a weak equivalence between spaces in hCW is a homotopy equivalence, the inclusion $\xymatrix{\Omega^{\mathcal C}(X)\ar[r]^-{\simeq}&\Omega X}$ is a homotopy equivalence, and the second claim follows.
\end{proof}

\section{Geometric Preliminaries}\label{prelims}


We introduce the main players in this paper and summarize their geometric properties. As indicated earlier, $M$ always stands for a path-connected simply connected Riemannian manifold, with a fixed Riemannian metric tensor $g$. There is a canonical metric $d_M$, coming from $g$,  obtained by stipulating that the distance between two points $d_M(p,q)$ is the greatest lower bound for the arc-lengths of piecewise smooth paths joining the points \cite{Milnor2}.  We also assume that $M$ is based, with $x_0\in M$ the chosen basepoint.

A path $(\gamma,r)$ in $M$ is  a \textit{geodesic} if the map  $\gamma:[0,r] \rightarrow M$  is $C^{\infty}$ with acceleration vector field that is identically zero. The velocity vector field must be constant and parallel along $\gamma$. 
This definition of geodesic uses the Levi-Civita connection on $M$ but luckily, there is a metric notion of geodesic, easier to work with, and which agrees with the former one. It is embodied in the following result (\cite{Milnor2}, Corollary 10.7): \textit{any smooth map $\gamma:[0,r] \rightarrow M$ that is length minimizer and has constant speed (both conditions together) must  satisfy the geodesic equation $\nabla_{\dot\gamma}\dot\gamma = 0$.}
We then define a \textit{minimal geodesic path} $(\gamma,r)$ to be any $C^\infty$-path that realizes the shortest distance between its endpoints and has constant speed $\|\dot\gamma\| $.
Once we fix the endpoints, we obtain the following definition (to be compared to
\cite{Milnor2}, Definition 1,\S10).

\begin{defn}\label{segment}\rm (Geodesic segment).  
A \textit{geodesic segment} between two distinct points $p,q$ is a minimal geodesic path $(\gamma,r)$ in $M$, such that $\gamma (0)=p$ and $\gamma (r)=q$. In other words, it is a {smooth} path between $p$ and $q$ such that
\[
r\|\dot\gamma(t)\|={d_M(p,q)}, \qquad \forall\, t\in[0,r].
\]
Note that when $r=1$, the speed is exactly the distance between the points, and when $r=d_M(p,q)$, the speed is $1$. In the latter case (unit speed), we say that the geodesic path is \textit{uniform}, in which case the segment between $p,q$ is parameterized by arclength on $[0,d_M(p,q)]\rightarrow M$. Clearly, the constant path at $p$ is the degenerate geodesic segment from \(p\) to \(p\).
\end{defn}


\bre\label{refining} A  very useful property of
minimizing geodesics is that \emph{every subsegment of a minimizing geodesic is
itself minimizing}. Subdividing a geodesic segment at an interior marked point produces two
geodesic segments.
If the original segment is \emph{uniform} in the terminology of
Definition \ref{segment} (unit speed), then the two resulting segments are also
uniform. 
Note that if
$x=\gamma(s)$, $0<s<r$,
is an interior marked point, then
\[
d_M(p,q)=d_M(p,x)+d_M(x,q),
\]
because \(x\) lies on the minimizing geodesic segment.
\ere

\begin{defn}\rm A ball $B$ in a Riemannian manifold is \textit{geodesically convex} if for any two points $ p,q\in B$, there is a  unique minimal geodesic segment joining $p$ to $q$, whose image lies in the ball. Moreover, and importantly, 
these geodesic segments depend differentiably on the endpoints (\cite{Milnor2}, \S10). 
\end{defn}

The \textit{convexity radius} $\operatorname{conv}(p)$ at a point $p$ is the largest $r$ such that the \textit{open ball} of radius $\rho$ centered at $p$; $B(p,\rho)$,  is geodesically convex, for all $\rho  <r$.  We will also refer to $B(p,\rho )$, $\rho < \operatorname{conv} (p)$, as being a \textit{geodesic ball} (so necessarily convex). We set
\[
\operatorname{conv}(M)
:= \inf_{p\in M}\operatorname{conv}(p) = 
\inf_{p\in M}
\sup\left\{ r>0 \;\middle|\;
B(p,r)\ \text{is  geodesically convex}
\right\},
\]
A manifold has \textit{positive convexity radius} if $\operatorname{conv}(M)>0$. This is always the case if $M$ is compact (see \cite{docarmo}, Proposition 4.2) or has bounded geometry. In the compact case, by (\cite{berger}, chapter 6, Proposition 95) 
$$0< \operatorname{conv}(M) \leq {\operatorname{inj}(M)\over 2}$$
where $\operatorname{inj}(M)$ is the injectivity radius of $M$. 
Note that for any $0<r< \operatorname{inj} (M)$,
$B(p,r)$ is diffeomorphic to an open ball. Note also that if $q\in B(p, \operatorname{inj}(M))$, then there is a unique minimizing geodesic segment from $p$ to $q$, whereas, when $q\in B(p, \operatorname{conv}(p))\subset B(p, \operatorname{inj}(M))$, there is a unique geodesic between \textit{any} $q$ and $q'$ in that smaller ball.  This is the meaning of geodesic convexity (also dubbed strong convexity in some books). 

We will set
\begin{equation}\label{delta}
\delta_M = \min \left\{1, {\operatorname{conv}(M)\over 4}\right\}
\end{equation}
For simplicity and where no confusion, we will write $\delta=\delta_M$.

Consider the uniform metric $d_\infty$ on the space of paths $\operatorname{Path}(M)$. This, we recall, is defined as follows:  given Moore paths $(\gamma_i,r_i) \in \operatorname{Path}(M)$, $i=0,1$, 
 $$
\begin{array}{ccc}
\tilde{\gamma_i}:\R_+ \rightarrow M & \tilde{\gamma}_i(t)=
\left\{
\begin{array}{cc}
\gamma_i(t) &t\leq r_i   \\
\gamma_i(r_i) &r_i\leq t 
\end{array}
\right. \end{array}
$$
we define
\begin{equation}\label{dinfinity}
d_\infty\big((\gamma_0,r_0),(\gamma_1,r_1)\big)
=
\sup_{t\in \R_+} d_M(\Tilde{\gamma_0}(t),\Tilde{\gamma_1}(t))+|r_0 -r_1|.
\end{equation}
The topology induced by this metric is called the uniform convergence topology.   In Section \ref{models}, we defined the Moore loop space \(\Omega X\) as a subspace of
$
C(\mathbb{R}_+,X)\times \mathbb{R}_+,
$ with the compact-open topology on the loop space factor. For metrizable \(X\), this subspace topology agrees with the topology induced by the metric \(d_\infty\) defined above (we refer to \cite[Proposition 5.3]{hu} whose proof for the case of ordinary loops adapts to this situation).

Two paths $(\gamma_0,r_0)$ and $(\gamma_1,r_1)$ of $M$
satisfying    $\gamma_0(0)=\gamma_1(0)=a$ and $\gamma_0(r_0)=\gamma_1(r_1)=b$,    are said to be homotopic relative to their endpoints if they lie in the same path component of the subspace 
\[
Path_{(a,b)}(M)=\{(\gamma,r)\in Path(M) \mid \gamma(0)=a \text{ and } \gamma(r)=b\} 
\]
 of $Path(M)$.

\begin{lem} \label{homotopicmaps}
Any two 
paths $(\gamma_0,r_0),  (\gamma_1,r_1)$ in $Path(M)$ such that $d_\infty((\gamma_0,r_0),(\gamma_1,r_1)\big) < \operatorname{conv} (M)$,  $\gamma_0(0)=\gamma_1(0)=a$ and $\gamma_0(r_0)=\gamma_1(r_1)=b$, are homotopic relative to their endpoints.
\end{lem}

\begin{proof} The proof is slightly easier for ordinary loops. For Moore loops, we need to adjust according to the parameter $r$. The steps are spelled out : 
\begin{enumerate}
    \item Step 1: $r_0=r_1=r$. For every $t\in \R_+$, $ \gamma_0(t)$ and $ \gamma_1(t)$ are within the convexity radius of $M$, and so are joined by a unique minimal geodesic  $(\alpha_t,1)$. Define $$h:[0,1] \rightarrow Path_{(a,b)}(M)$$   by $h(s) = (h_s,r)$, where $h_s:[0,r] \rightarrow M$ is given by $h_s(t)= \alpha_t(s) $. Then $h$ is continuous, $h(0)= (\gamma_0,r_0)$ and $h(1)= (\gamma_1,r_1)$. It follows that 
    $(\gamma_0,r_0)$ and $(\gamma_1,r_1)$ are homotopic relative to the endpoints.
    \item Step 2: general case. The relation "homotopic relative to the endpoints" is an equivalence relation, therefore we may assume that $r_0 \leq r_1$. Let $(\beta_0  ,r_1)$ be the path in $M$ given by 
    \[
     \beta_0(t)=
\left\{
\begin{array}{cc}
\gamma_0(t) &t\leq r_0   \\
\gamma_0(r_0) &r_0\leq t \leq r_1
\end{array}
\right.
\]
Then clearly, $(\gamma_0,r_0)$ and $(\beta_0,r_1)$ are homotopic relative to the endpoints. Moreover  
$\tilde{\beta_0}=\tilde{\gamma_0}$, therefore
$$ d_\infty\big((\beta_0,r_1),(\gamma_1,r_1)\big) \leq d_\infty\big((\gamma_0,r_0),(\gamma_1,r_1)\big)< \operatorname{conv}(M).$$
By Step 1, $(\beta_0,r_1)$ and $(\gamma_1,r_1)$ are homotopic relative to the endpoints. It follows that 
 $(\gamma_0,r_0)$ and  $(\gamma_1,r_1)$ are homotopic relative to the endpoints.
\end{enumerate}
\end{proof}


\section{The Broken Geodesic Models}\label{piecewisemodels}

From now on, we use the following notations and subspace inclusions
$$\Omega M\hookrightarrow PM\hookrightarrow \operatorname{Path}(M) \subset C^0(\mathbb{R}_{+},M)\times \mathbb{R}_{+}.$$
where $\operatorname{Path}(M)$ and $PM$ are, respectively, the space of all paths and the space of all based paths, as previously defined. 

We start with a definition. 

\begin{defn}\label{broken}\rm (Broken Geodesic).
 A \textit{broken geodesic path}  in $M$ is a path
$(\gamma,r) $ in $M$ such that for a finite partition $t_0=0\leq t_1\leq \cdots \leq t_n=r$, and all $1\leq i\leq n$, the composite of the affine reparameterization $[0,t_i-t_{i-1}]\rightarrow [t_{i-1}, t_i]$ with the restriction of $\gamma$ 
$$\gamma_i : [0, t_i-t_{i-1}]\rightarrow [t_{i-1},t_i]\rightarrow M$$
defines a geodesic segment $(\gamma_i, t_i-t_{i-1})$ from $\gamma (t_{i-1})$ to $\gamma (t_i)$. We write $x_i := \gamma (t_i)$ in $M$ and call these images the ``marked points'' of the broken path.  Note that in the subdivision, $t_i$ and $t_{i+1}$ can be the same and so marked points can coincide. Note also that under the product \eqref{productpath}, we can succinctly represent the broken geodesic as a concatenation of segments
$$(\gamma, r) = (\gamma_1, r_1)*\cdots *(\gamma_{n},r_{n})$$
where $r_i=t_i-t_{i-1}$.
\end{defn}

\begin{defn}\label{marking}\rm (Marking) We put an equivalence relation on the set of tuples in $M$ according to
$$(x_0,\ldots, x_n)\sim (x_0,\ldots, \hat x_i,\ldots, x_n)\ \ \hbox{if}\ x_i=x_{i+1}\ ,\ 0\leq i < n$$
 We recall that the standard notation $\hat x_i$  means removal of the $i$-entry. An equivalence class is written ${\bf  m} =[x_0,\ldots, x_n]$, and a \textit{marking} is a representative with  $x_i\neq x_{i+1}$ for $0\leq i < n$. 
Any broken geodesic with subdivision $0=t_0\leq t_1\leq \cdots \leq t_n=r$ as in Definition \ref{broken} defines a marking
$[\gamma (t_0),\ldots, \gamma (t_n)]$. All broken geodesics come with a marking but this is not unique, it depends on the subdivision of $[0,r]$, and so does not determine the broken geodesic as a parameterized path in $M$. The model $H_\infty$ that we will introduce shortly is meant to remedy to this shortcoming precisely \S\ref{marked}.
\end{defn}

Let $M$ be a Riemannian manifold with its canonical metric $d_M$. We define the space of broken geodesic loops
\begin{equation}\label{omegapg}
\Omega^{bg}(M) = \{(\gamma,r)\in\Omega M\ |\ \gamma\ \hbox{is a broken geodesic, $\gamma (0)=\gamma (r)= x_0$}\}
\end{equation}
This is a proper subspace of $\Omega M$, the space of all based loops at $x_0$, with the compact-open topology. The usual loop product in $\Omega M$ restricts to a product on $\Omega^{bg} (M)$.
The subspace $\Omega^{bg} (M)$  depends on the Riemannian metric tensor $g$ we choose on $M$. It turns out however that its weak homotopy type does not depend on that metric (Theorem \ref{pgapprox}).  

\subsection{Uniform paths and the small model}\label{sigmacompact} A model is said to be \textit{small} if it is $\sigma$-compact.
A \textit{uniform} broken geodesic is a broken geodesic
$(\gamma,r)$ with marking ${\bf  m} = [z_0,\ldots, z_n]$, $\gamma (t_i)=z_i$, such that 
$$||\dot\gamma (t)|| = 1\ \ \ \forall\ t\in (0,r)\setminus \{t_1,\ldots, t_{n-1}\}.
$$

If $\gamma$ is a uniform broken geodesic, then $L(\gamma )=r$ and 
$t_{i+1}-t_i  = {d_M(z_{i+1},z_i)}$.
We will write $\Omega^{ubg} (M)\subset \Omega^{bg} (M)$ the subspace of all uniform broken geodesics. Both $\Omega^{bg} (M)$ and $\Omega^{ubg} (M)$ inherit the standard $H-$space structure defined on $\Omega M.$

The following is the main result of this section, and gives our  first set of $H$-space models.

\begin{thm}\label{pgapprox} The inclusions
   $\Omega^{ubg}(M)\hookrightarrow  \Omega^{bg}(M)\hookrightarrow \Omega M$ are homotopy equivalences. 
\end{thm}

\begin{proof}
Define $P^{bg}(M)$ the space of all broken geodesics starting at $x_0$, ending anywhere. This is clearly a contractible subspace of $PM$ since the standard ``path shrinking'' contraction of $PM$ is a linear map on segments, and so restricts to $P^{bg}(M)$.
Any cover of $M$ by geodesically convex balls is distinguished for this collection of paths (see \S\ref{collection}). Indeed, over every geodesically convex ball $U$, the evaluation $ev^{-1}(U)\rightarrow U$ has an evident lifting function which consists of assigning to every $(\gamma,r)\in ev^{-1}(U)$, and every path $\lambda$ in the base starting at $\gamma (r)$, the following ``path of paths'' in $ev^{-1}(U)$  $$(\gamma_t,r+d_t)\ \ ,\ \ \gamma_0=\gamma\ \ ,\ \  d_t=d_M(\gamma (r),\lambda (t))$$
where $\gamma_t = \gamma *\alpha_t$, and $\alpha_t$ is the unique minimal geodesic from $\gamma (r)$ to $\lambda (t)$, parameterized over $[0,d_M(\gamma (r),\lambda (t))]$.
The claim is now immediate by Proposition \ref{distinguished}.
The exact same argument applies to the endpoint projection $P^{ubg}(M)\rightarrow M$, where $P^{ubg}(M)$ is the subspace of $P^{bg}(M)$ of uniform based paths. 

To upgrade the weak homotopy equivalences we just obtained to homotopy equivalences, it
suffices to show, by Whitehead's theorem, that all spaces have the
homotopy type of CW-complexes. This is however an immediate consequence of Proposition~3 of \cite{schon}. It states that if
$
F\longrightarrow E\longrightarrow B
$
is a Hurewicz fibration and \(E,B\in\mathrm{hCW}\), then
\(F\in\mathrm{hCW}\). Each space in the above diagram is the fiber of
a Hurewicz fibration over \(M\) whose total space is the corresponding
based path space:
$P^{ubg}(M),
P^{bg}(M)$ and $
PM$ respectively. All these total spaces are contractible and hence belong
to \(\mathrm{hCW}\). Since the manifold \(M\) also has the homotopy
type of a CW-complex, Proposition~3 of \cite{schon} implies that all
the fibers belong to \(\mathrm{hCW}\).
\end{proof}


\section{Combinatorial Approximations: the Monoid Model}\label{marked}

As mentioned in the introduction, our combinatorial models consist of tuples of points in $M^n$, for all $n$, related by identifications. 
The idea is to ``reconstruct'' the space of loops from the markings of closed geodesics, up to homotopy. 

We first define the following subspace of tuples, with $x_0\in M$ a fixed basepoint
\begin{equation}\label{markedloops}  
\tilde G_k = \{(x_0, x_1, \ldots x_{k-1},x_0)\in M^{k+1}\ |\ d_M(x_i, x_{i+1}) \leq \delta,\ 0\leq i\leq k-1 , d(x_{k-1},x_0)\leq \delta \}
\end{equation}
As in Definition \ref{marking}, we can consider the quotient construction
\label{Hn}
\begin{equation}\label{Hn}
H_n=\Big(\coprod_{0\le k\le n}\widetilde G_k\Big)\big/\sim
 \end{equation} 
where $\sim$ is the equivalence relation generated by the relation below
\begin{equation}\label{relation1}(x_0, \ldots, x_i, \cdots, x_k=x_0) \sim (x_0,\ldots,
\hat x_i, \cdots, x_k=x_0)\ \ \ \ \ \quad\text{whenever }x_i=x_{i+1}
\end{equation}  
An element in $H_n$ is a marking ${\bf m} = [x_0,x_{1}, \ldots , x_{k-1},x_0]$, and any such marking gives rise to a well-defined broken geodesic loop in $M$ based at $x_0$ by taking the concatenation of the unique geodesic segments $\sigma_i$'s between consecutive entries $x_i,x_{i+1}$, parameterized by arclength over $[0,d_M(x_i,x_{i+1})]$ (i.e. uniform), as done earlier. 
Observe that two tuples that differ by a repeated entry give rise to the exact same broken geodesic. We therefore obtain a well-defined map
\begin{equation}\label{imageonto}
\phi_n: H_n\lrar\Omega^{ubg}_\delta (M,n)\hookrightarrow \Omega M\ \ ,\ \ \phi_n(\bf m) = \gamma_{\bf m}
\end{equation}
where, by definition,
$\Omega^{ubg}_\delta (M,n)$
is the image of $\phi_n$ in
$\Omega^{ubg}(M)$, and this consists of all uniform broken geodesics with at most $n$ marked point.
The map $\phi_n$ is continuous by an adequate version of Lemma \ref{continuous}, where $H_n$ has the quotient topology of the identification space, and the target $\Omega M$ has the compact open topology. The map onto $\Omega^{ubg}_\delta(M,n)$ is surjective but not injective.

There are inclusions
$H_n\hookrightarrow H_{n+1}$, and juxtaposing tuples yields maps $\mu_{n,m}$ that render the following diagram strictly commutative.
\begin{equation}\label{Hmap}
\xymatrix{
H_n\times H_m\ar[d]^{\phi_n\times\phi_m}\ar[r]^{\mu_{n,m}}& H_{n+m}\ar[d]^{\phi_{n+m}}\\
\Omega^{ubg}_\delta(M,n)\times 
\Omega^{ubg}_\delta(M,m)\ar[r]^{\ \ \ \ \ \ *}&
\Omega^{ubg}_\delta(M,n+m)
}
\end{equation}
This turns the direct limit $H_\infty := \operatorname{colim} H_n$, with its colimit topology, into a topological monoid (see Appendix B where the continuity of the multiplication is established). Note that there are inclusions
$\Omega^{ubg}_\delta(M,n)\subset\Omega^{ubg}_\delta(M,n+1)$ obtained by introducing a new marked point as an interior point of any segment.

\begin{lem}
    There is a homeomorphism 
    $\Omega^{ubg}(M)=\bigcup_{n\geq 1}\Omega^{ubg}_\delta(M,n)$, where the union is endowed with the subspace topology of $\Omega M$.
\end{lem}

\begin{proof}
    This is an obvious consequence of the fact that any broken geodesic loop in $\Omega M$ can be refined to a broken geodesic with successive marked points a distance at most $\delta$ of each other by Remark \ref{refining}.
\end{proof}

The maps $\phi_n$ induce a map $\phi'_\infty: H_\infty\rightarrow \Omega^{ubg}_c(M)\rightarrow\Omega^{ubg}(M)$,  where $\Omega^{ubg}_c(M)=\operatorname{colim}_n \Omega^{ubg}_\delta(M,n)$ in the colimit topology. The second composite is the identity map which is continuous in this direction. This is a map of $H$-spaces and we write
\begin{equation}\label{phinfinity}
 \phi_\infty : H_\infty\lrar \Omega^{ubg}(M)\hookrightarrow \Omega M \end{equation}
the map into the full loop space.
The following is the main result of this section. 

\begin{thm}\label{happrox}
 The map $\phi_\infty : H_\infty\rightarrow \Omega M$ is a homotopy equivalence of $H$-spaces.
In particular, $H_\infty$ is an $H$-space model for $\Omega M$ (\S\ref{models}). 
\end{thm}

\begin{proof} Define the following subspace of marked paths
\begin{equation}\label{markedpaths}
\tilde P_k = \{(x_0,x_{1},\ldots, x_k)\in M^{k+1}\ |\ d_M(x_i,x_{i-1})\ \leq\delta, 1\leq i\leq k\}
\end{equation}
where again $x_0=v_0$ is a fixed basepoint,
and the associated ``path-space''
\begin{equation}\label{phn}
PH_n= \bigsqcup_{k=0}^n \tilde P_k/_{\sim}
\end{equation}
where the relation $\sim$ is as defined in \eqref{relation1}. Note that $\tilde G_k\subset \tilde P_k$, and of course $H_n\subset PH_n$.
We will show that $PH_n$ is contractible.
For \({\bf m}=[x_0,\ldots,x_k]\), $k\leq n$, put
\[
d_i=d_M(x_{i-1},x_i),\qquad s_i=\sum_{j=1}^i d_j,\quad s_0=0,
\]and let \(\gamma_{\bf m}:[0,s_k]\to M\) traverse the associated broken geodesic at unit speed. Define
\[
H_t({\bf m})=
\bigl[\gamma_{\bf m}((1-t)s_0),\ldots,\gamma_m((1-t)s_k)\bigr].
\]The distance between consecutive new markings satisfies
\[
d_M\!\left(\gamma_{\bf m}((1-t)s_{i-1}),
            \gamma_{\bf m}((1-t)s_i)\right)
\leq (1-t)(s_i-s_{i-1})
=(1-t)d_i\leq\delta.
\]
The inequality holds because distance is at most the length of the intervening subpath. If two adjacent markings coincide, deleting one removes a zero-length
segment and leaves both the arclength-parametrized path and the
parameters of all remaining markings unchanged. Thus \(H_t\) respects
the relation \(\sim\) and descends to a continuous contraction of $PH_n$.

Next, we consider the projection maps
$PH_n\lrar M$ sending 
$[x_0,\ldots, x_k]\mapsto x_k$, $k\leq n$. These maps are continuous and induce a continuous map
$\pi : PH_\infty\rightarrow  M$. 
We claim in the next Lemma  that this projection is a fibration. Since it is, and since it maps to the path fibration over $M$, and since $PH_\infty$ is contractible in its colimit topology, the map between fibers $\phi_\infty$ is a weak homotopy equivalence as claimed. To go from weak equivalence to a genuine equivalence, we again resort to \cite{schon} to show that $H_\infty\in$ hCW. Indeed, since the base and total space of $\pi: PH_\infty\rightarrow M$ are in hCW, then so is the fiber $H_\infty$.
\end{proof}

\begin{lem}\label{lem:free loop fibration}
The projection $\pi : PH_\infty \to\ M$ is locally a fibration and
a surjection. Thus the projection $\pi$ is a fibration by the Dold-Hurewicz theorem (Theorem \ref{Dold}).
\end{lem}

\begin{proof}
Choose an open cover $\mathcal U$ of $M$ by geodesically convex
sets of diameter strictly less than $\delta$.
We first verify surjectivity. Given $y\in M$, choose a path
$\alpha\colon I\to M$ from $x_0$ to $y$ and a subdivision
$
0=t_0<t_1<\cdots<t_k=1
$
such that
$
d_M\bigl(\alpha(t_{i-1}),\alpha(t_i)\bigr)<\delta$, $(1\leq i\leq k)$.
Then
$$
[v_0,\alpha(t_1),\ldots,\alpha (t_k)=y]\in PH_\infty
$$
and this element projects to $y$.

Next we show that
$\pi^{-1}(U)\rightarrow U$ is a Hurewicz fibration, for any $U$ chosen in this open cover.
The proof follows that of Proposition \ref{distinguished} and relies on showing the existence of slices or local lifting functions (see \S\ref{collection}).
Let 
$m=[x_0,x_1,\dots,x_k=u]\in \pi^{-1}(U)$, and $\lambda:I\to U,\ \lambda(0)=u.$
Since \(U\) is geodesically convex, for each \(t\in I\) there is a unique geodesic segment from \(u\) to \(\lambda (t)\) lying in \(U\). Define 
$$
\widetilde\lambda_m(t):=[v_0,x_1,\dots,x_k=u,\lambda(t)].
$$
This is a path in \(\pi^{-1}(U)\) lifting \(\lambda\) and starting at \(m\). Note that $d(u,\lambda (t))<\delta$ for all $t$ since  diam$(U)<\delta$. The assignment $
(m,\lambda)\mapsto \widetilde\lambda_m$
is continuous by the continuous dependence of geodesic segments on endpoints. Hence \(\pi^{-1}(U)\to U\) admits a lifting function, so it is a fibration. Since this holds for every open set of the cover, the map $\pi$ being obviously surjective, it must be a fibration by Theorem \ref{Dold}.
\end{proof}

Finally, we check the following assertion made in Theorem \ref{main}.

\begin{lem}\label{sigmacompact}
\(\Omega^{ubg}(M)\) is $\sigma$-compact
\end{lem}

\begin{proof}
Write 
$\displaystyle 
\Omega^{ubg}(M)=\bigcup_{n\ge 1}\Omega^{ubg}_\delta(M,n)$ as introduced in \eqref{imageonto}. If $M$ is compact, then each \(\Omega^{ubg}_\delta(M,n)\) is compact
 by Corollary \ref{compact},
and \(\Omega^{ubg}_\delta(M)\) is $\sigma$-compact by construction.
In general, any finite dimensional manifold \(M\) is \(\sigma\)-compact, 
then so is every finite product \(M^{n+1}\). Since
\(\widetilde G_n\) in \eqref{markedloops}  is a closed subspace of \(M^{n+1}\), it is also \(\sigma\)-compact. Its quotient under
the identification of repeated consecutive points is therefore again \(\sigma\)-compact.
As a final consequence,  \(\Omega^{ubg}(M)\) is a countable union of \(\sigma\)-compact spaces, and hence is
itself \(\sigma\)-compact.
\end{proof}


\section{Combinatorial Approximations: the Group Model}\label{groupmodel}

We will now refine our $H_\infty$ model and replace it by a group model $G_\infty$. The proof of this will be much more demanding since the contractibility of the corresponding path space is much less straightforward.

The starting point is the family of spaces of ``marked loops'' $\tilde G_k$ defined in \eqref{markedloops}, so we can set
\begin{equation}\label{Gn}
 G_n := \bigsqcup_{k=0}^n \tilde G_k/_{\approx} 
 \end{equation}
where $\approx$ is the equivalence relation generated by the relation $\sim$ in \eqref{relation1}, corresponding to the removal of $x_i$ when $x_i=x_{i+1}$, and of the ``Milnor relation''
\begin{equation}\label{relation2}(x_0, \ldots, x_i, \cdots, x_k) \sim_M (x_0,\ldots,
\hat x_i, \cdots, x_k)\ \ \ \ \ \quad\text{whenever }x_{i-1}=x_{i+1}
\end{equation}  
The relation $\sim_M$ is  the one defined in \cite{Milnor} for $M$ a connected simplicial complex. We set $G_\infty = \operatorname{colim}_nG_n$.

As for the $H$-space model, we consider a space of ``marked paths'' in this combinatorial setup by setting
$$PG_n= \bigsqcup_{k=0}^n \tilde P_k/_{\approx}$$
where $\tilde P_k$ is as defined in \eqref{markedpaths}, with $n=\infty$ allowed and $PG_\infty$ topologized as a direct limit. The coordinate projection maps 
\begin{equation}\label{pik}\pi_k: \tilde P_k
\to\ M \ \ ,\ \ \pi_k(x_0, \cdots,
x_k)=x_k
\end{equation}
are continuous and compatible with
the equivalence relation $\approx$, they induce a continuous map 
\begin{equation}\label{eval}
\pi: PG_\infty \longrightarrow \ M\ ,\ 
[x_0,\ldots, x_k]\mapsto x_k
\end{equation}
The following Theorem is the main result of this section (and of \cite{bc}).

\begin{thm}\label{reducedapprox}
Assume that $M$ is a path-connected Riemannian manifold of positive convexity radius, with \(\delta\) satisfying \eqref{delta}. Then
\begin{enumerate}
  \item $G_\infty$ is a topological group, and $\pi$ is the
projection in a principal $G_\infty$-bundle.
  \item The space $PG_\infty$ is contractible.
\end{enumerate}
In particular, $G_\infty$ is a group model for $\Omega M$ and $BG_\infty$ is weakly equivalent to $M$.
\end{thm}

As in the case of \cite{Milnor}, the most subtle point is to prove (2). Part (1) is covered by Lemma \ref{principalGbundle}, while part (2) is covered by Proposition \ref{contractibility2} in a separate section.

For each $j,k\geq 0$, consider the pairing
$$\mu: \tilde G_k\times \tilde P_j \lrar\ \tilde P_{j+k}$$ given by concatenation as follows
\begin{equation}\label{concat2}
\mu((x_0, x_1,\ldots, x_{k-1}, x_0), (y_{0}=x_0,y_1\cdots, y_j)) = (x_0,\cdots,
x_{k-1},x_{0},y_1,\cdots, y_j)
\end{equation}
Observe that the endpoint of the path does not change under this composition by an element in $\tilde G_k$. The maps $\mu$ induce an action after passing to colims.

\begin{lem} \label{lem:continuity}
Concatenation induces a continuous map $\mu: G_\infty\times  PG_\infty
\to\ PG_\infty$ and a commutative
diagram.
$$\xymatrix{G_\infty\times G_\infty\ar[d]\ar[r]^{\mu_|}&G_\infty\ar[d]\\
 G_\infty\times PG_\infty\ar[r]^\mu&PG_\infty}$$
The restriction $\mu_|$ makes $G_\infty$ into a topological group
with identity given by $[x_0]$, while the map $\mu$ gives an action of
$G_\infty$ on $PG_\infty$. The orbit space of the action $PG_\infty/
G_\infty$ is homeomorphic to $M$.
\end{lem}

\begin{proof} The argument for the continuity of $\mu$, and its restriction $\mu_|$, is analogous to the argument outlined in Appendix B, so it is skipped. 
On the other hand, the endpoint projection $PG_\infty\rightarrow M$ is constant on orbits, which are precisely the fibers, so it induces a quotient homeomorphism. This gives a continuous bijection $PG_\infty/G_\infty\rightarrow M$ which is a homeomorphism by Lemma \ref{principalGbundle} next.
\end{proof}

As in \eqref{productpath}, we compose paths from left to right. In other words, if $\gamma_1$ and $\gamma_2$ are two composable geodesic paths, meaning the endpoint of $\gamma_1$ is the starting point of $\gamma_2$, we write $\gamma_1\gamma_2$ their composition or product. Conjugation of a loop $g$ by a path $\gamma$ is written
$g^\gamma := \gamma g\gamma^{-1}$.

\begin{lem}\label{principalGbundle} (Theorem \ref{reducedapprox}, (1)).
The projection $\pi: PG_\infty \to\ M$ is the projection in a
principal $G_\infty$-bundle.
\end{lem}

\begin{proof}
Fix a point $p$ in $M$ and a geodesic ball $U_p$ of diameter strictly less that $\delta$, centered at $p$. Let $e_p$ denote a chosen basepoint in $\pi^{-1}(p)$ which we write as
$$e_p = [x_0, x_1,\ldots, x_{k-1},p].$$ 
For $g= [x_0,x_1,
\cdots, y_{n-1},x_0]\in G_\infty$, and any $x\in U_p$, the
element $ g\cdot e_p \cdot [p,x]$ lies in $\pi^{-1}(U_p)$ since $d_M(p,x)<\delta$.
This defines a ``trivialization'' map 
$$\phi_p: G_\infty\times U_p  \to\ \pi^{-1}(U_p)\ \ ,\ \ \phi_p(g,x)= g\cdot e_p \cdot [p,x].$$ 
satisfying
$\pi\circ \phi_p(g,x) = x$. We show it is a homeomorphism by defining a
map in the other direction which composes with $\phi_p$ to the identity.

Let $e$ be any point in $\pi^{-1}(U_p)$. It is represented by $e = [x_0, y_1,\ldots, y_k]$
with $\pi(e)= y_k$ in $U_p$. There is a unique minimal
geodesic from $p$ to $y_k$, since both are in $U_p$. We can
consider their product 
$$ e\cdot [\pi(e),p]\cdot e_p^{-1} = 
[x_0,y_1, \cdots, y_{k-1},y_k] \cdot [y_k,p]\cdot [p,x_{k-1}, \cdots, x_1,x_0].$$ 
This is an element of $G_\infty$.
We can then define 
$\theta :\pi^{-1}(U_p) \to\ G_\infty $ by the formula
$\theta (e) = e\cdot [\pi(e),p] \cdot e_p^{-1} $. By the previous
paragraph, this function is continuous. There is an associated
continuous function $$\theta_p :\pi^{-1}(U_p) \to\ 
G_\infty\times U_p  \ \ \ ,\ \ \ \theta_p(e) =  (\theta (e), \pi(e)).$$ 
This is a map over $U_p$, and
it is straightforward to check that $\theta_p$ is an inverse to $\phi_p$.
It now remains to see that the locally trivial projection
$\pi: PG_\infty \to M$ is that of a principal bundle, and this lies
in the transition functions.

Let $U_p$ and $U_q$ be two convex neighborhoods of $p$ and $q$ respectively,
such that $U_p \cap U_q$ is non-empty, and let $x\in U_p
\cap U_q$. The element $[p,x,q]$ is defined. Define the
function
$$g_{p,q}: U_p \cap U_q \to\ G_\infty$$ by the formula
$$g_{p,q}(x) = e_p\cdot [p,x,q]\cdot e_q^{-1} $$
The following formula holds for the transition map
$$(\theta_p\circ\varphi_q)(g,x)
=\bigl(g\,g_{q,p}(x),x\bigr)
=\bigl(g\,g_{p,q}(x)^{-1},x\bigr)$$
and the cocycle conditions are satisfied: $g_{p,p}=id$, $g_{p,q}(x) = g_{q,p}^{-1}(x)$, and
on triple intersections $g_{p,q}(x)g_{q,r}(x) = g_{p,r}(x)$.
These are the principal bundle conditions, hence the proof is complete.
\end{proof}

\begin{rem}\rm  (Functoriality) Let $M$ and $N$ be two based Riemannian manifolds. Every basepoint-preserving Riemannian isometry
$f: (M,x_0)\longrightarrow (N,y_0)$
induces an isomorphism of topological monoids
$H_\infty(M)\cong H_\infty(N)$, 
and an isomorphism of topological groups
$G_\infty(M)\cong G_\infty(N).$
This is true because under an isometry,
$\operatorname{conv}(M)=\operatorname{conv}(N)$, and thus
$\delta_M=\delta_N$,
so the defining distance conditions are preserved. The defining equivalence relations are also preserved, so the map $f_*: G_\infty(M)\rightarrow G_\infty(N)$, defined on equivalence classes by $f_*([x_0,\ldots, x_n])=[f(x_0),\ldots, f(x_n)]$, is a well-defined bijective homomorphism.
\end{rem}


\section{Contractibility of the marked path space $PG_\infty$}\label{contractibilityproof}

We first prove the contractibility of the marked path space $PG_\infty$ and then finish off the proof of Theorem \ref{main}. The argument of the contractibility of $PH_\infty$ in the proof of Theorem \ref{happrox} does not apply here because the contraction respects repeated points \eqref{relation1} but fails to respect the additional backtracking relation \eqref{relation2}.
New constructions are needed
and these are formalization of ideas in \cite{Milnor}. Our context is slightly more general than Milnor's since we do not have to restrict to CW-complexes. Throughout, $M$ denotes a path-connected Riemannian manifold
of positive convexity radius, and $\delta=\delta_M$ is as
defined in \eqref{delta}.

\begin{lem}\label{fourthcontractible}
    Assume we have a surjective quotient map of NDR pairs $(B,A)\rightarrow (Y,X)$, and assume $B\setminus A$ is homeomorphic to $Y\setminus X$. If $A,B$ and $X$ are contractible, then $Y$ must also be contractible.
\end{lem}

\begin{proof}
We have a map of pairs $(B,A)\rightarrow (Y,X)$ which induces a homeomorphism on complements. Because the map is a quotient map, this implies that  
$B/A$ is homeomorphic to $Y/X$. On the other hand, because $(B,A)$ is an NDR pair, and $A$ contractible, it follows that $B/A\simeq B$ which is contractible. So $Y/X$ is contractible. For the same reasons, $Y/X\simeq Y$, so $Y$ is contractible.
\end{proof}

\begin{lem}\label{nested} 
Let
$
X_0\subset X_1\subset X_2\subset \cdots$
be a nested sequence of contractible subspaces, and each inclusion is a cofibration. Then $\hbox{colim}_{n} X_n=\bigcup_{n\ge 0} X_n$ is contractible.
\end{lem}

\begin{proof}
The ordinary colimit agrees up to homotopy with the homotopy colimit, and the homotopy colimit is equivalent to a homotopy colimit of a constant diagram at a point, which is just a point.
\end{proof}

\begin{lem}\label{nerve}
    Let $A_1,\ldots, A_n$ be spaces such that all intersections are cofibrant and contractible, including the spaces $A_i$. We assume that $\bigcap A_i\neq\emptyset$. Then $\bigcup A_i$ is contractible.
\end{lem}

\begin{proof} By the homotopy nerve lemma, \(\bigcup_i A_i\) has the homotopy type of the nerve \(N(\mathcal A)\).
Since \(\bigcap_i A_i\neq\varnothing\), the full vertex set \(\{1,\dots,n\}\) is a simplex of \(N(\mathcal A)\).
Therefore \(N(\mathcal A)\) is a cone, hence contractible.
\end{proof}

\begin{lem}\label{quotient-ndr} (Pushing NDR data along a quotient map).
Let $A\subset B$ be a closed subspace, and let $q:B\longrightarrow X$
be a surjective quotient map. Assume that the restriction
$q|_{B\setminus A}:B\setminus A \longrightarrow X\setminus q(A)$
is a homeomorphism, and
that $A$ is saturated (i.e. $
q^{-1}(q(A))=A)$.
If $(B,A)$ is an NDR-pair, then $\bigl(X,q(A)\bigr)$ is also an NDR-pair.
\end{lem}

\begin{proof}
Since $(B,A)$ is an NDR-pair, there exist (by definition) maps
\[
u:B\to [0,1],\qquad H:B\times I\to B
\]
such that $A=u^{-1}(0)$ and
\begin{align*}
H(b,0)&=b \qquad \text{for all } b\in B,\\
H(a,t)&=a \qquad \text{for all } a\in A,\ t\in I,\\
u(b)<1&\Longrightarrow H(b,1)\in A.
\end{align*}
We use this to construct NDR-data on $X$. We have the diagram
$$\xymatrix{
A\ar[d]^{q_|}\ar[r]&B\ar[d]^q\ar[dr]^u\\
q(A)\ar[r]&X\ar[r]^v&[0,1]}$$
for some induced map $v$, $v^{-1}(0)=q(A)$. Since $q$ is a quotient map, $v$ is continuous if $u$ is. We need the homotopy $K:X\times I\to X$ with the right properties to show $(X,q(A))$ is NDR. It is direct to check that 
\[
K:X\times I\to X,\qquad K(x,t)= \begin{cases}q\bigl(H(b,t)\bigr),&b\ \hbox{unique preimage of $x$ if $x\in X\setminus q(A)$}\\
x&\hbox{if $x\in q(A)$}
\end{cases}
\]
is the desired homotopy (well-defined and continuous).
\end{proof}

\begin{lem}\label{pushlaongpaths} (Extending homotopies along geodesic segments).
Let $\overline B(x,\delta)$ be a closed geodesic ball centered at $x\in M$, with $\delta$ as earlier defined \eqref{delta}, and let $\gamma: I\rightarrow M$, $\gamma (0)=x$ be a piecewise smooth path starting at $x$. Then there is a homotopy $H^\gamma : \overline B(x,\delta)\times I\rightarrow M$ so that $H^\gamma (y,0)=y$ and 
\[
d_M\bigl(H^\gamma(y,t),\gamma(t)\bigr)\le \delta
\qquad\text{for all }(y,t)\in \overline B(x,\delta)\times I.
\]
If the path $\gamma$ is in $ \overline B(x,\delta)$, then the homotopy stays in the ball at all times.
\end{lem}

\begin{proof} For each \(y\in \overline B(x,\delta)\), let \(\rho(y):=d_M(x,y)\). By assumption, there is a unique
minimizing geodesic from \(x\) to \(y\); let
\[
\sigma_y:[0,\rho(y)]\longrightarrow M
\]
denote this geodesic, parametrized by arc length, so that
\[
\sigma_y(0)=x,\qquad \sigma_y(\rho(y))=y.
\]
The assignment \(y\mapsto \sigma_y\) depends continuously on \(y\) as indicated \S\ref{prelims}.
Let $x$ travel along $\gamma$. For a given $y$, we can make $y$ travel along $\sigma_y$ so as to keep the distance between this $y(t)$ and $\gamma (t)$ within $\delta$. 
If $y$ reaches $x$ (it may not if $d(x,y)>L(\gamma)$), 
then we make $y$ follow
$\gamma$ with constant arc-length delay equal to 
$d(x,y)$.
This keeps the distance to 
$\gamma (t)$ at most $\delta$ at all times.

Writing up the key details of this construction is given below. Let $
L:=\operatorname{Length}(\gamma)$. Since \(\gamma\) is piecewise smooth, there is an arclength parameterization of the path
$\widetilde\gamma:[0,L]\longrightarrow M$ such that
$
\gamma(t)=\widetilde\gamma(\ell(t)),$
where
$\ell(t):=\operatorname{Length}(\gamma|_{[0,t]}).$
Thus \(\ell:I\to [0,L]\) is continuous and nondecreasing, with \(\ell(0)=0\).

For each \(y\in \overline B(x,\delta)\), let
$\sigma_y$ 
be the unique minimizing geodesic from \(x\) to \(y\) as above.
One now defines \(H^\gamma\) by
\[
H^\gamma(y,t)=
\begin{cases}
\sigma_y\bigl(\rho(y)-\ell(t)\bigr), & \text{if } \ell(t)\le \rho(y),\\[1ex]
\widetilde\gamma\bigl(\ell(t)-\rho(y)\bigr), & \text{if } \ell(t)\ge \rho(y).
\end{cases}
\]
This has the properties described in the first paragraph.
\end{proof}

\begin{rem}\label{endofpath}\rm
    In Lemma \ref{pushlaongpaths}, notice that if \(x\) does not move, that is, if
$
\gamma$ is constant at $x$, then the length traveled by \(\gamma\) up to time \(t\) is
$
\ell(t)=\operatorname{Length}(\gamma|_{[0,t]})=0$,
for all $t\in I$.
Therefore, 
\[
H^\gamma(y,t)=\sigma_y\bigl(\rho(y)-\ell(t)\bigr)
=\sigma_y(\rho(y))
=y
\qquad\text{for all }(y,t)\in \overline B(x,\delta)\times I.
\]
and the path $
H^\gamma(y,-):I\longrightarrow M$ is constant. In other words, if \(x\) does not move along \(\gamma\), then \(y\) does not move under the homotopy \(H^\gamma\).
\end{rem}

Using all five lemmas above, we are ready to prove our main claim.

\begin{prop}\label{contractibility2} (Theorem \ref{reducedapprox}, (2)) 
$PG_\infty$ is contractible.
\end{prop}

\begin{proof}
The starting point is \eqref{markedpaths} which is rewritten below
$$\tilde P_k = \{(x_0,x_{1},\ldots, x_k)\in M^{k+1}\ |\ d_M(x_i,x_{i-1})\ \leq\delta, 1\leq i\leq k ,\ x_0=v_0 \}$$
For $k\geq 2$, define the following subspaces
\begin{itemize}
    \item $R_i = \{(x_0,x_1,\ldots, x_k)\in \tilde P_k\ |\ x_{i-1}=x_i\}$,\ for $i\leq k$
    \item $Q_j = \{(x_0,x_1,\ldots, x_k)\in \tilde P_k |\ x_{j+1}=x_{j-1}\}$,\ for $j\leq k-1$
    \item $\tilde D_k\subset\tilde P_k$ is the subspace of all ``degenerate configurations'' 
    \begin{equation}\label{degenerate}\tilde D_k = R_1\cup\cdots\cup R_k\cup Q_1\cup\cdots\cup Q_{k-1}
    \end{equation}
\end{itemize}
Note that each $R_i$ and each $Q_j$ is a closed subspace of $\tilde P_k$, and according to Lemma \ref{ndrpair} next,  the pairs $(\tilde P_k, R_i)$ and $(\tilde P_k, Q_j)$ are NDR pairs, $\forall i,j$. Their intersections are closed cofibrations as well. It follows that $i_k: \tilde D_k\hookrightarrow \tilde P_k$ is a cofibration (\cite{lellig}, Corollary 2). 

By construction, we consider for fixed $k\geq 1$, the commuting 
diagram of spaces
\begin{equation}\label{pushoutdiagramG}
\xymatrix{\tilde D_k\ar@{^(->}[r]^{i_k}\ar[d]^{q_|}&\tilde  P_{k}\ar[d]^q\\
PG_{k-1}\ar@{^(->}[r]^{j_{k-1}}&PG_{k}}
\end{equation}
where $q_|$ is the restriction of $q$ to $\tilde D_k$. We observe that $\tilde P_k\setminus \tilde D_k\cong PG_k\setminus PG_{k-1}$. The map $i_k$ is a cofibration as already mentioned, and $q$ is a surjective quotient map. Since $\tilde D_k$ is saturated, by Lemma \ref{quotient-ndr}, $j_{k-1}$ is also a cofibration. We know already that $\tilde P_k$ is contractible and the same argument of Milnor (his proof of Lemma 3.6 in the simplicial case) shows that the various intersections $R_{i_1}\cap\cdots\cap R_{i_r}\cap Q_{j_1}\cap\cdots\cap Q_{j_s}$ are contractible, non-empty. Note that the intersection of all the $R_i$'s and $Q_j$'s is $\{v_0\}$. By Lemma \ref{nerve}, it follows that $\tilde D_k$ is contractible. Since both $\tilde P_k$ and $\tilde D_k$ are contractible, then using Lemma \ref{fourthcontractible}, we can infer that $PG_k$ is contractible if $PG_{k-1}$ is. This becomes an inductive argument, starting with $PG_0 =[v_0]$. All the $PG_k$'s are thus contractible, consequently  $PG_\infty$ is also contractible by Lemma \ref{nested}.
\end{proof}

\begin{lem}\label{ndrpair}
    $(\tilde P_k, R_i)$ and $(\tilde P_k, Q_j)$ are NDR pairs, for all $1\leq i\leq k$, $1\leq j\leq k-1$.
\end{lem}

\begin{proof}
We will prove that both $R_i$ and $Q_j$ are strong deformation retracts of $\tilde P_k$, which at least for metric spaces, implies our claim. Indeed, if $A$ is a closed subspace in a metric space $X$ that is a deformation retraction via a homotopy $H$, one can obtain the NDR data for $(X,A)$ by setting $u: X\rightarrow [0,1]$, 
$$
u(x)= \min (1, d(x,A))$$
and with the same $H$
(see proof of Lemma \ref{quotient-ndr}).

Fix \(i\), with \(1\le i\le k\). We define a strong deformation retraction
\[
H:\widetilde P_k\times I\longrightarrow \widetilde P_k
\]
onto \(R_i\). In words, this homotopy leaves the marked points $x_j$ for $j\leq i-1$ stationary, while it moves the $x_j, j\geq i+1$ along geodesic segments between the points so that consecutive distances remain less than $\delta$.
More precisely, let 
$p=(x_0,x_1,\dots,x_k)\in \widetilde P_k$, 
and let \(\eta_i^p:I\to M\) be the minimizing geodesic from \(x_i\) to \(x_{i-1}\) lying in \(\overline B(x_i,\delta)\). We define recursively, for \(j=i,\dots,k-1\),
\begin{equation}\label{iteratively}
\eta_{j+1}^p(t):=H^{\eta_j^p}(x_{j+1},t),
\end{equation}
where \(H^{\eta_j^p}\) is the homotopy supplied by Lemma~\ref{pushlaongpaths}, applied with base point \(\eta_j^p(0)\) and path \(\eta_j^p\).

We then set
\[
H(p,t):=(x_0,\dots,x_{i-1},\eta_i^p(t),\eta_{i+1}^p(t),\dots,\eta_k^p(t)).
\]
By Lemma~\ref{pushlaongpaths}, for every \(j\ge i\) and every \(t\in I\),
$
d_M(\eta_{j+1}^p(t),\eta_j^p(t))\le \delta.$
Since the coordinates \(x_0,\dots,x_{i-1}\) are fixed, it follows that
$H(p,t)\in \widetilde P_k$ for all $(p,t)\in \widetilde P_k\times I.$

Also,
$
H(p,0)=p,
$
because \(\eta_i^p(0)=x_i\) and recursively \(\eta_{j+1}^p(0)=x_{j+1}\).
Moreover, at time \(t=1\), $
\eta_i^p(1)=x_{i-1},
$
since \(\eta_i^p\) is the geodesic from \(x_i\) to \(x_{i-1}\), so that $H(p,1)\in R_i$.

Finally, $R_i$ is fixed by the homotopy since if \(p\in R_i\), then \(x_i=x_{i-1}\), so \(\eta_i^p\) is the constant path. By Remark~\ref{endofpath}, this implies that \(\eta_{i+1}^p,\dots,\eta_k^p\) are all constant as well.
The continuity of \(H\) follows from the continuous dependence of the geodesic segment \(\eta_i^p\) on \(p\).

The proof that $Q_j$ is a strong deformation retract of $\tilde P_k$ is identical, only that the initial choice of path is different. With the same $p$ as above, we let $\eta^p_{j+1}$ be the unique minimal geodesic starting at $x_{j+1}$ and ending at $x_{j-1}$. This geodesic exists in $\overline B (x_j,\delta)$ because both $x_{j+1}$ and $x_{j-1}$ are in $\overline B(x_j,\delta)$, and the ball is geodesically convex. One now defines the paths $\eta_r$ for $r\geq j+1$ iteratively as in \eqref{iteratively}, and the homotopy
\[
H(p,t):=(x_0,\dots,x_{j-1},x_j,\eta_{j+1}^p(t),\eta_{j+2}^p(t),\dots,\eta_k^p(t)).
\]
This homotopy takes 
the marked point $x_{j+1}$ to $x_{j-1}$ along the path $\eta_{j+1}^p(t)$ and pulls the remaining marked points along so that consecutive distances remain less than $\delta$, precisely as was done in the previous deformation retraction. Note that at all times, $d_M(x_j,\eta_{j+1}^p(t))\leq\delta$ because the path $\eta_{j+1}^p(t)$ is in the geodesic ball. This process depends continuously on $p$ and $\eta_{j+1}^p$.
\end{proof}

\begin{cor}
    The quotient map $H_\infty\rightarrow G_\infty$ is a homotopy equivalence.
\end{cor}

\begin{proof}
The quotient map \(PH_\infty\to PG_\infty\) is a map over \(M\)
between contractible total spaces. The induced map on the fibers is
the quotient map \(H_\infty\to G_\infty\). The long exact sequences
of the two fibrations therefore show that this map is a weak homotopy
equivalence. Since both \(H_\infty\) and \(G_\infty\) lie in
\(\mathrm{hCW}\), Whitehead's theorem implies that it is a homotopy
equivalence.
\end{proof}


\section{Modeling the Free Loop Space}\label{freeloop}

We keep our notation $\Omega M$ for the Moore loop space and we write $LM$ for the analog, unbased space of Moore loops. If we start with a group model $G$ of $\Omega M$, as discussed in \S\ref{models}, then we can reconstruct $LM$ as follows.
Let $E$ be a contractible space with a left free $G$ action, and let $G$ act by itself by conjugation
$$G\times G\rightarrow G\ \ ,\ \ 
(h,g)\mapsto g^h := hgh^{-1}$$ 
This is a left action.
The notation for $G$ as a $G$-space via this conjugation action is written $G^{ad}$.
We can then construct the associated space
\begin{equation}\label{egad}E\times_GG^{ad} := E\times G/_\sim\ \ ,\ \ (e,g)\sim (he, hgh^{-1})\ ,\ h\in G
\end{equation}
This is the quotient by the diagonal action of $G$. 

\begin{lem}\label{recognition2} Let $X$ be a well-pointed simply connected space, $G\rightarrow E\rightarrow X$ a principal $G$-bundle, $E$ contractible, and both $X,G$ in hCW, then there is a homotopy equivalence $ LX\simeq E\times_GG^{ad}$.
\end{lem}

\begin{proof} This is standard, with the starting point being that
$LBG$ is weakly equivalent to $EG\times_GG^{ad}$ (see \cite{kss}, Appendix). By the universality of $EG$, there is a $G$-map $E\rightarrow EG$ covering the classifying map
$c: X\rightarrow BG$. This map induces a homotopy equivalence $E\times_GG^{ad}\rightarrow EG\times_GG^{ad}$.
Putting it all together, we obtain the sequence of weak homotopy equivalences $$
\xymatrix{LX\ar[r]^-{Lc}&LBG\ar[r]^-{\simeq}&
EG\times_GG^{ad}&E\times_GG^{ad}\ar[l]_-\simeq}$$ 
Since $X$ and $G$ are in hCW, all other spaces are  and so all maps are homotopy equivalences.
\end{proof}

Next, consider the space of ``marked unbased loops'' $L_\infty$ obtained as follows. For $k=0$, $L_0=M$, and for each  $k\geq 1$, define
$$L_k = \{(x,x_1, \cdots , x_k, x)\in M^{k+2} \ |\ \ d_M(x_i,x_{i+1})\leq \delta , d_M(x,x_1)\leq\delta, d_M(x_k,x)\leq\delta \}$$
then set
$L_\infty= \bigsqcup_{k\geq 0}L_k/_\approx$ to be the quotient space under the identification $\approx$ which is a combination of both \eqref{relation1} and \eqref{relation2}. As before, elements in $L_\infty$ are equivalence classes of tuples written $[x,x_1,\ldots, x_k,x]$ for some $k$, and there is a projection
 $$\pi : L_\infty\rightarrow M
\ \ ,\ \ 
[x,x_1,\ldots, x_k,x]\longmapsto x
 $$
which is continuous.

\begin{lem} The projection map $\pi : L_\infty\rightarrow M$ is a fiber bundle, with fiber $G_\infty$. 
\end{lem}

\begin{proof}
The argument is standard (as in Lemma \ref{principalGbundle}) and the key point is the existence of a trivializing cover by geodesic balls. More precisely,  
Choose an open cover $\{U_p\}_{p\in \Omega}$ of $M$ by geodesically
convex balls, where $p$ denotes the center of $U_p$
and $\Omega\subset M$ is a (countable) set.
Choose once and for all a marked
piecewise-geodesic path $\alpha_p\in PG_\infty$ from $x_0$ to $p$.
Observe that $$G_\infty = \{[x_0, x_1,\ldots, x_k,x_0]\ , \hbox{for some $k$},\ d(x_i, x_{i+1})\leq\delta\}\ \ \ \ 
(\hbox{see \S\ref{groupmodel}})$$ is the fiber over a basepoint $x_0\in M$.  For any $x\in U_p$, consider the unique minimal geodesic $\sigma_x$ from $p$ to $x$, and build the map
 $$\pi^{-1}(U_p)\rightarrow G_\infty\times U_p\ \ ,\ \ 
 \gamma = [x, x_1,\ldots, x_k,x] \longmapsto (\gamma^{\alpha_p*\sigma_x}, x)$$
 where $\gamma^{\alpha_p*\sigma_{x}} : = (\alpha_p*\sigma_x)*\gamma*(\alpha_p*\sigma_x)^{-1}$ is the conjugated loop starting and ending at $x_0$ (recall our convention for the product of loops: $\ell_1*\ell_2$ means we apply $\ell_1$ first then $\ell_2$). This map 
is a homeomorphism and gives the desired local trivialization.
\end{proof}

Next, we will set $E=PG_\infty$, the space of all marked paths and $G=G_\infty$, and construct a map 
$\Theta : PG_\infty \times_{G_{\infty}}G_{\infty}^{ad}\lrar L_\infty$ as follows. Take an element $[\gamma, g]$ in the adjoint bundle, where $\gamma = [x_0,x_1,\ldots ,x_n]$ and $g = [x_0,y_1,\ldots, y_k, x_0]$, for some $n,k>0$. Conjugation gives an element in $L_\infty$;
$$g^{\gamma^{-1}} := \gamma^{-1} * g * \gamma = [x_n,\ldots, x_1, x_0, y_1,\ldots, y_k, x_0,x_1\ldots, x_n] \in L_\infty$$
Pictorially, this is the class of a piecewise geodesic path that starts at the endpoint of $\gamma$, goes to $x_0$, then loops around along $g$, and comes back to $x_n$ via $\gamma$. This path does not depend on the representative class since if $[\gamma, g] = [h\gamma , hgh^{-1}]$, $h\in G_\infty$, then 
$$g^{\gamma^{-1}} = (hgh^{-1})^{(h\gamma)^{-1}}$$
Part of the main proposition below shows it is a homeomorphism.
This Proposition is the main result of this section.

\begin{prop}\label{modelfreeloops}  $L_\infty$ has the homotopy type of $LM.$
\end{prop}

\begin{proof} 
We have a well-defined map 
$$\Theta : PG_\infty \times_{G_{\infty}} G_{\infty}^{ad}\lrar L_\infty\ \ ,\ \ [\gamma, g]\mapsto g^{\gamma^{-1}}$$
which is a homeomorphism.  To see this, for every $x\in M$, choose a geodesic path $\alpha_x= [x_0,x_1,\ldots, x_n,x]\in PG_\infty$ from $x_0$ to $x$ (using the local sections in the proof of Lemma \ref{principalGbundle}). The inverse $\Gamma$ of $\Theta$ can now be defined as follows
$$\Gamma (\gamma ) = [\alpha_x, \gamma^{\alpha_x}]$$
This we can check is independent of the choice of $\alpha_x$, and is continuous.
The homotopy equivalence follows from Lemma \ref{recognition2}.
\end{proof}

\section{Modeling Maps from a two-dimensional Complex}\label{surfaces}

For $X$ and $Y$ path-connected and well-pointed spaces, we will write $\bmap (X,Y)$ (resp. $\map (X,Y)$) the space of based maps (resp. unbased maps) from $X$ to $Y$, with the compact-open topology. Fix a topological group $G$ and consider the product
of commutators map $$\chi_g: G^{2g} \lrar G\ \ \ ,\ \ \ 
\chi_g(x_1,x_2, \cdots,x_{2g}) = [x_1,x_2]\cdot [x_3,x_4] \cdots
[x_{2g-1},x_{2g}]$$ 
where $[a,b] = aba^{-1}b^{-1}$.
Let $\xi_{g}(G)$ denote the homotopy theoretic fiber of the map
$\chi_g$. There is a natural action of $G$ on this homotopy fiber and the following (known) result is stated in \cite{bc}.

\begin{prop}\label{StoBG}
Assume that $M$ is a simply-connected  Riemannian manifold with $\operatorname{conv}(M)>0$. Denote by $S_g$ a compact Riemann surface of genus $g\geq 1$, and set $G=G_\infty$.
\begin{enumerate}
  \item There is a homotopy equivalence
$\xi_g(G) \lrar \bmap (S_g,M).$
\item There is a homotopy equivalence
$PG_\infty \times_G \xi_g(G) \lrar \map (S_g, M).$
\end{enumerate}
\end{prop}

The preceding proposition appears in \cite{grusal}, where it is stated
for any group model \(G\) of \(\Omega M\), although no proof is given
for the unbased case. Since the present paper builds on \cite{bc}, we
show that Proposition \ref{StoBG} is a special case of a more general
result applying to finite CW complexes of dimension at most two, rather
than only to the circle and Riemann surfaces.

More precisely, suppose \(X\) has one \(0\)-cell, \(n\) \(1\)-cells, and \(m\) \(2\)-cells attached by relator words
\[
r_1,\dots,r_m\in F_n, \ \ n,m\geq 0
\]
where $F_n$ is the free group on $n$-generators corresponding to the fundamental group of the $1$-skeleton.
These determine a \textit{relator map}
$R:G^n\longrightarrow G^m$ given by
\begin{equation}\label{relator} R(g_1,\dots,g_n)=\bigl(r_1(g_1,\dots,g_n),\dots,r_m(g_1,\dots,g_n)\bigr).
\end{equation}
Denote the ordinary based path space by
$$P_0(G^m)=\{\gamma:I\to G^m\mid \gamma(1)=e\}
$$
Define
$\xi_R(G)$ to be the pullback of the relator map in this diagram
\begin{equation}\label{pullback}\xymatrix{
\xi_R(G)\ar[r]\ar[d]&P_0(G^m)\ar[d]^{ev_0}\\
G^n\ar[r]^R&G^m}
\end{equation}
In other words, $\xi_R(G)=G^n\times_{G^m}P_0(G^m)$ is the homotopy fiber of $R$. Note that \(G\) acts on \(G^n\) and \(G^m\) by conjugating every entry, and it fixes the basepoints. Every relator word is conjugation-equivariant, so the map \(R\) is \(G\)-equivariant, and hence the pullback \(\xi_R(G)\) inherits a strict \(G\)-action. 

The following theorem is the main result of this section. The case \(m=0\), \(n=1\) recovers the free loop space model \S\ref{modelfreeloops}, while \(m=1\), \(n=2g\), with \(R\) given by the product
of commutators, recovers Proposition \ref{StoBG}.

\begin{thm}\label{XtoBG} Let $G=G_\infty$ be our group model for a compact simply-connected Riemannian manifold $M$ and \(X\) a finite connected CW complex with one \(0\)-cell, obtained
by attaching \(m\geq 0\) two-cells to \(\bigvee_{j=1}^{n}S^{1}\).
Then \(\xi_R(G)\) is a model for $\bmap (X,M)$. In other words, there is a map
$$u: \bmap (X,M)\lrar \xi_R(G)$$
which is a weak homotopy equivalence. Moreover, under the induced conjugation action discussed above, there is a weak homotopy equivalence
\[
\map (X,M)\simeq PG_\infty\times_{G_\infty} \xi_R(G).
\]
\end{thm}

\begin{proof} 
In this proof, $\Omega M$ will mean everywhere the space of ordinary loops (non-Moore). We can assume wlog that the $1$-skeleton of $X$ is a wedge $\bigvee^nS^1$,  to which we are attaching $m$ two-cells. The attaching-map parametrizations 
$
C:W_m=\bigvee_{i=1}^{m}S^1
\longrightarrow
W_n=\bigvee_{j=1}^{n}S^1;
$
so that
\[
\operatorname{Map}_*(W_n,M)
\longrightarrow
\operatorname{Map}_*(W_m,M),
\qquad u\longmapsto u\circ C.
\]
corresponds 
under the identifications
\(\operatorname{Map}_*(W_n,M)\cong(\Omega M)^n\) and
\(\operatorname{Map}_*(W_m,M)\cong(\Omega M)^m\), to the relator map
\[
\varphi(\mathbf{x})
=\bigl(r_1(\mathbf{x}),\ldots,r_m(\mathbf{x})\bigr),
\]
Let's choose the reduced mapping-cone model
\begin{equation}\label{conepresentation}
X=W_n\cup_C\operatorname{Cone}(W_m),
\qquad
W_k=\bigvee_{j=1}^{k}S^1,
\end{equation}
Choosing the standard path-space model
\(P_0(\Omega M)=\{F:I\to\Omega M\mid F(1)=e_{x_0}\}\),
we can write the homotopy fiber of $\varphi: (\Omega M)^n\rightarrow (\Omega M)^m$ in the form
\begin{equation}\label{hofibphi}
\hbox{hofib} (\varphi)
=
\{({\bf x},\lambda)\in (\Omega M)^n\times P_0((\Omega M)^m)\mid \lambda (0)= \varphi ({\bf x}),\ \lambda(1)=e\},
\end{equation}
We claim that with the reduced cone presentation \eqref{conepresentation}, there is a homeomorphism
\begin{equation}\label{themapq}
q:\operatorname{Map}_*(X,M)\xrightarrow{\ \cong\ }
\bigl\{(\mathbf{x},\lambda)\mid
\lambda(0)=\varphi(\mathbf{x}),\ \lambda(1)=e\bigr\}
=\operatorname{hofib}(\varphi).
\end{equation}
To see this, write $X$ in the form of the reduced mapping-cone model
\[
X= W_n\cup_C \operatorname{Cone}_*(W_m),
\qquad
\operatorname{Cone}_*(W_m)
=\frac{W_m\times I}{W_m\times\{1\}\,\cup\,\{*\}\times I},
\]
where \([a,0]\) is identified with \(C(a)\in W_n\).

For a based map \(f:X\to M\), let \(u_f=f|_{W_n}\) and define
$$
\lambda_f(s)(a)=f([a,s]), a\in W_m,\ s\in I.$$
Thus \(\lambda_f(s)\) is a based map \(W_m \to M\) for every \(s\).
Moreover,
$\lambda_f(0)=u_f\circ C,
\qquad
\lambda_f(1)=e,
$
where \(e: W_m\to M\) is the constant map at \(x_0\). Identifying
\(\operatorname{Map}_*(W_n,M)\) with \((\Omega M)^n\) and
\(\operatorname{Map}_*(W_m,M)\) with \((\Omega M)^m\), we set
\[
q(f)=(\mathbf{x}_f,\lambda_f) \in \hofib (\varphi)
\]
This map is a homeomorphism whose inverse glues this
data to form a based map \(X\to M\).

Let
$
\psi:\Omega M\longrightarrow G
$
be the boundary map of the principal bundle
$G\rightarrow PG_\infty\rightarrow M
$.
By Corollary~2.3, $\psi$ is an $A_\infty$-map (in particular an $H$-map) and a weak homotopy
equivalence. 
The $A_\infty$-structure of $\psi$ provides a homotopy between
the evaluation of any fixed word before and after applying $\psi$.
Applying this simultaneously to the relator words
$r_1,\ldots,r_m$ gives a specified homotopy
$H:R\psi^n\simeq\psi^m\varphi$.
In particular, the diagram
\begin{equation}\label{diagram2}
\xymatrix{
(\Omega M)^n \ar[r]^-{\varphi} \ar[d]_{\psi^n}
&
(\Omega M)^m \ar[d]^{\psi^m}
\\
G^n \ar[r]_-R
&
G^m
}
\end{equation}
is homotopy commutative which then induces a map of homotopy fibers
$\eta:\operatorname{hofib}(\varphi)\longrightarrow\operatorname{hofib}(R)=\xi_R(G)$. This map is given by
\begin{equation}\label{mapeta}\eta ({\bf x},\lambda) =  (\psi^n {\bf x}, H_{\mathbf x}*(\psi^m\circ\lambda))
\end{equation}
where $H$ is the fixed homotopy induced from $\psi$ and $H_{\bf x}$ is the trace of $x$, i.e. $H_{\bf x}=H(-,{\bf x})$, $x\in (\Omega M)^n$, so that
$H_{\mathbf{x}} * (\psi^m\circ\lambda)$
is a path in \(G^m\) from
$R\bigl(\psi^n(\mathbf{x})\bigr)$
to \(e\).
Since $\psi$ is a homotopy equivalence, so is $\eta$.  

We can then define
\[
u:=\eta\circ q:
\map_*(X,M)\longrightarrow\xi_R(G).
\]
This is a homotopy equivalence, which proves the first assertion.

We now turn to the unbased mapping space. We use the classical principle that we can reconstruct a fibration from its strong homotopy holonomy action 
\cite{stasheff}. A holonomy is not just a map as in \eqref{holonomy}, but consists of a collection of maps forming an \(A_\infty\)-action, together with higher coherent homotopies among the iterated actions, parametrized by the associahedra. 

Consider the two fibrations over
$M$
\begin{equation}\label{twofibrations}
\map_*(X,M)
\longrightarrow
\map(X,M)
\xrightarrow{\ ev\ }
M
\end{equation}
and
\begin{equation}\label{associatedfibration}
\xi_R(G)
\longrightarrow
PG_\infty\times_G\xi_R(G)
\longrightarrow
M .
\end{equation}
The strategy is to compare their strong holonomy actions.
In \eqref{associatedfibration}, $G$ acts diagonally 
as in \eqref{egad}. If we write explicitly
\[
\xi_R(G)=\hofib (R)
=
\{({\bf y},\gamma )\in G^n\times P_0(G^m) \mid \gamma (0)=R({\bf y}),\ \gamma (1)=e\}
\]
then the action
$a:G\times\xi_R(G)\longrightarrow\xi_R(G)$ is given on an element in the
homotopy fiber $z=(\mathbf g,\lambda)$ by
\begin{equation}\label{holonomya}
a\bigl(h,(\mathbf g,\lambda)\bigr)
=
\bigl(
h\mathbf g h^{-1},
h\lambda h^{-1}
\bigr).
\end{equation}
A choice of holonomy of the associated bundle
\eqref{associatedfibration} is $h_{\mathrm{ass}}(\gamma,z) =
a\bigl(\psi(\gamma)^{-1},z\bigr).$
Indeed, lift $\gamma$ to a path in $PG_\infty$ beginning at a chosen
point $e_0$ of the fiber over the basepoint.  Its endpoint is
$\psi(\gamma) e_0$.  Hence the lift
$t\longmapsto[\widetilde\gamma(t),z]$
in the associated bundle ends at
$[\psi(\gamma)e_0,z]
=
[e_0,\psi(\gamma)^{-1}z]$.

If we write $\bar a(h,z)=a(h^{-1},z)$, we  claim that
there is a diagram of strong homotopy actions
\begin{equation}\label{holonomydiagram}
\xymatrix{
\Omega M\times\map_*(X,M)
\ar[r]^-{h_{ev}}
\ar[d]_{id\times q}
&
\map_*(X,M)
\ar[d]^q\\
\Omega M\times\hofib(\varphi)
\ar[r]^-{h_{\varphi}}\ar[d]_{\psi\times \eta}&\hofib(\varphi)\ar[d]_{\eta}
\\
G\times\xi_R(G)
\ar[r]_{\bar{a}}
&
\xi_R(G)
}
\end{equation}
A square diagram of strong-homotopy actions would mean that the diagram
commutes up to \emph{coherent homotopies}: the two composites are joined
by a specified homotopy compatible with the actions, together with
higher homotopies for all iterated products and action maps, satisfying
the \(A_\infty\) (Stasheff) coherence relations.
Let's now see how this applies to diagram \eqref{holonomydiagram}. 

Consider the top part of the diagram .
The holonomy action \(h_{ev}\) is described in Appendix C: the action of $\gamma\in\Omega M$ on the restriction of $f$ to the $1$-skeleton is conjugation by Lemma \(\ref{one-skeleton-conjugation}\), so the first coordinate of $q(h_{ev}(\gamma,f))$ is $\gamma^{-1}\mathbf{x}_f\gamma$, which is precisely $h_{\varphi}(\gamma,q(f))$ on the first coordinate. Similarly on the second coordinate, and so we define the strong-homotopy
action \(h_\varphi\), including its higher structure maps, by
transporting \(h_{\mathrm{ev}}\) through the homeomorphism \(q\).
Consequently \(q\) intertwines the actions strictly at every level and the diagram commutes as strong homotopy actions.
 
Turning to the bottom square, we compare \(h_\varphi\) with the strict action \(\bar a\) through
\(\eta\). 
To that end, we need make the strong-homotopy action 
$h_\varphi$ more explicit. A point in $\hofib (\varphi)$ is of the form $({\bf x},\lambda)$ \eqref{hofibphi}. The map $h_\varphi$ on the first coordinate is represented by conjugation,
$
(\gamma,\mathbf x)\longmapsto \gamma^{-1}\mathbf x\gamma,
$ so it is strict, 
but on the ``path coordinate'' (second factor) 
conjugation of an element in $P_0(\Omega M)^m$ by $\gamma\in\Omega M$ does not land in that space, only up to homotopy!
To correct this, one must insert the coherence paths 
$
\varphi(\gamma^{-1}\mathbf x\gamma)
\simeq
\gamma^{-1}\varphi(\mathbf x)\gamma
$, and
$
\gamma^{-1}e\gamma\simeq e
$. After doing so, we write $\lambda^\gamma$ the resulting path in
\(P_0((\Omega M)^m)\).
Homotopy lifting  extends this map to a strong-homotopy (\(A_\infty\))-action
\(h_\varphi\).

Writing
\(z=(\mathbf x,\lambda)\), and recalling the map $\eta$ from \eqref{mapeta}, the two composites around the square are
\[
\eta h_\varphi(\gamma,z)
=
\left(
\psi^n(\gamma^{-1}\mathbf{x}\gamma),
\,
H_{\gamma^{-1}\mathbf{x}\gamma}
*
(\psi^m\circ\lambda^\gamma)
\right)
\]
and
\[
\bar a\bigl(\psi(\gamma),\eta(z)\bigr)
=
\left(
\psi(\gamma)^{-1}\psi^n(\mathbf{x})\psi(\gamma),
\,
\psi(\gamma)^{-1}
\left(
H_{\mathbf{x}} * (\psi^m\circ\lambda)
\right)
\psi(\gamma)
\right).
\]
On the first coordinate, this comes down to constructing a coherent homotopy
between the maps
$\Omega M\times\Omega M\longrightarrow G$
given by
$(\gamma, x)\longmapsto
\psi(\gamma^{-1} x\gamma)
$
and
$
(\gamma, x)\longmapsto
\psi(\gamma)^{-1}\psi(x)\psi(\gamma)$ (one factor at a time). This is clearly provided by $\psi$ and its $A_\infty$ structure. 
On the second coordinate, we must compare, compatibly with the homotopy on the first coordinate, the two maps
$
\Omega M\times\hofib(\varphi)\longrightarrow P_0(G^m)
$
given by
$
(\gamma,(\mathbf x,\lambda))
\longmapsto
H_{\gamma^{-1}\mathbf x\gamma}
*
(\psi^m\circ\lambda^\gamma)
$
and
\[
(\gamma,(\mathbf x,\lambda))
\longmapsto
\psi(\gamma)^{-1}
\bigl(H_{\mathbf x}*(\psi^m\circ\lambda)\bigr)
\psi(\gamma).
\]
The first-level
homotopy in the coherent system relating the two strong-homotopy
actions is constructed using \(H\) and the structure of \(\psi\) by relating the iterated word and conjugation maps. The higher structure maps of \(\psi\) provide the
higher homotopies. This data is enough to make \((\psi,\eta)\) a morphism of
strong-homotopy actions. We skip the details.

To conclude, by Stasheff's
classification of fibrations in terms of strong-homotopy \cite{stasheff, stasheff2},
the two fibrations \eqref{twofibrations} and \eqref{associatedfibration} have weakly homotopy equivalent total spaces
\[
\map(X,M)
\simeq
PG_\infty\times_G\xi_R(G).
\]
This proves the second assertion.
\end{proof}

\appendix
\section{Non-Moore loops}\label{ordinary}

Generally in applications, people work with ordinary loops (i.e. non-Moore loops) and paths parameterized over the unit interval.
The unbased path space is $Path(X)=C^0([0,1],X)$, and  the loop spaces are $\Omega X$ and $LX=X^{S^1}=C^0(S^1,X)$.
Loop approximation results can be stated in this 
case as well, as done in this paper, but the details are more tedious due to the constant reparameterizations required for most arguments. This is the  reason why we opted early on to work with Moore loops.
In this Appendix we collect further useful results, stated for non-Moore loops, for the benefit of users. 

From now on, $M$ is always a compact Riemannian manifold, $\Omega M$ and $PM$ are the ordinary (i.e. non-Moore) loops and paths, respectively.

As in \S\ref{prelims}, a geodesic segment between two distinct points
$p,q$ is a minimal geodesic path
\[
\gamma\colon [a,b]\longrightarrow M,
\qquad \gamma(a)=p,\quad \gamma(b)=q,
\]
parameterized so that
\[
(b-a)\|\dot\gamma(t)\|=d_M(p,q),
\qquad t\in[a,b].
\]
Note that when $[a,b]=[0,1]=I$, the speed is exactly the distance between the points. 

We can also define an (ordinary) broken geodesic
path $\gamma : [0,1]\rightarrow M$ to be a path such that for a finite partition $t_0=0\leq t_1\leq \cdots \leq t_n=1$, and all $1\leq i\leq n$, the restrictions
$$\gamma_i :[t_{i-1},t_i]\rightarrow M$$
are geodesic segments from $\gamma (t_{i-1})$ to $\gamma (t_i)$, for all $i$. As in \S\ref{piecewisemodels}, we write $x_i := \gamma (t_i)$ in $M$ the marked points.

\vskip 10pt
\noindent{\bf Properties of geodesic segments}: The following properties hold.

\begin{enumerate}[(i)]
\item  Let $f: [c,d]\rightarrow [a,b]$ be an affine reparameterization $f(x)=\alpha x + \beta$, and let $\gamma : [a,b]\rightarrow M$ be a geodesic segment, then the composition $\gamma\circ f: [c,d]\rightarrow M$ is again a geodesic segment, since
the composite is still constant speed and length-minimizing.
\item The restriction of a geodesic segment $[a,b]\rightarrow M$ to $[a,t]\subset [a,b]$ is still a minimal geodesic, and
$\displaystyle 
L\big(\gamma|_{[a,t]}\big)=\frac{t-a}{b-a}\,d_M(p,q).$  In particular, if $\gamma : [a,b]\rightarrow M$ is a geodesic segment, and $a<c<b$, then the restrictions $\gamma_{|[a,c]}: [a,c]\rightarrow M$ and $\gamma_{|[c,b]}$ are geodesic segments as is (i.e. no change of parameterizations needed).
\item  The parameterization of a geodesic segment is, up to an affine change of parameter, an arclength parameterization. Say $d_M(p,q)=c$, $\gamma : [0,1]\rightarrow M$ a geodesic segment between $\gamma (0)=p,\gamma (1)=q$, and $\lambda : [0, c]\rightarrow M$ is the arclength parameterization of that path. Then $\gamma (t)=\lambda (ct)$. Note that $\lambda$ has unit speed. 
\item Let $\gamma:[0,1]\to M$ be a broken geodesic such that
$||\dot\gamma (t)|| = c$ is constant for all $t\in (0,1)\setminus \{t_1,\ldots, t_{n-1}\}.$ Then $\displaystyle c=L(\gamma) = \int_0^1\|\dot \gamma (t)\|dt$. This is a direct
consequence of (ii).\\
\end{enumerate}

The above leads to the following definition.

\begin{defn}\rm (Uniform broken geodesics, non-Moore loops). A broken geodesic $\gamma : [0,1]\rightarrow M$ is uniform if on $(0,1)\setminus \{t_1,\ldots, t_{n-1}\}$ it is parameterized by the constant speed $\|\dot\gamma(t)\| = L(\gamma )$.
\end{defn}

The next proposition shows that the broken geodesic paths (and loops) are dense in all paths (resp. loops) in the compact-open topology.

\begin{prop}\label{density} $\Omega^{bg} (M)$ is a dense subset of $\Omega M$.
\end{prop}

\begin{proof} 
Let $\delta$ be as defined in \eqref{delta}, and define the path space $\operatorname{Path}(M)=M^I$. Any element
$\gamma \in Path^{bg}(M)$ has a subdivision $\ t_0=0\leq t_1\leq \cdots \leq t_n=1$  such that $d_M(\gamma (t_i),\gamma (t_{i+1}))\leq\delta,\ 0\leq i\leq n-1\}$. In other words, $\gamma$ has a marking $\ (x_0, x_1,\ldots, x_n)$   satisfying $ d_M(x_i,x_{i+1})\leq\delta,\ 0\leq i\leq n-1\}$.  We will write $\operatorname{Path}^{bg}(M) = \operatorname{Path}^{bg}_\delta(M)$ when we insist that any broken geodesic path has this type of marking. 

Define $\operatorname{Path}^{bg}_\delta(M,n)$ to be the subspace of $\operatorname{Path}^{bg}(M)$ of all broken geodesic paths that have $n$ marked points $x_1,\ldots, x_n$ or less, and consecutive distances $d(x_{i-1},x_i)\leq\delta$. 
We need to show that
\begin{eqnarray*}
\operatorname{Path}^{bg} (M) 
&=&\bigcup_{n\geq 1}\operatorname{Path}^{bg}_\delta(M,n)
\end{eqnarray*}
is dense in $\operatorname{Path} (M)$.

Let $\gamma : [0,1]\rightarrow M$  be any path in $M$ and let $\epsilon > 0$. The family 
$$\left\{\gamma^{-1}\left(B(p,\min\left\{\dfrac{\epsilon}{2}, \frac{\delta}{2} \right\})\right), \ p\in M \right\}$$ is an open cover of $[0,1]$ which is compact. It has a Lebesgue number $\ell > 0$.
Let $n$ be a positive integer satisfying $\dfrac{1}{n}< \ell .
$ 
Any subset of diameter ${1\over n}$ or less is in some element of the cover, by definition of the Lebesgue number. It follows that
for any $i\in \{0,1,\dots,n-1\}$,  
the entire image $\displaystyle \gamma \left(\left[\frac{i}{n}, \frac{(i+1)}{n}\right]\right)$ lies in $B (p,\min\left\{\frac{\epsilon}{2}, \frac{\delta}{2}\right\} )$ for some $p\in M.$  
Consider the marking 
${\bf m} = [x_0,\ldots, x_n]$ with
$\displaystyle x_i = \gamma \left({i\over n}\right), 0\leq i\leq n$. The consecutive marked points $x_i,x_{i+1}$ are in the same geodesic ball, so we can write $\sigma_i$ 
the unique geodesic segment between them. Notice that the entire segment is also in that ball by convexity.
Concatenating these segments gives a broken geodesic $\sigma$ in $\operatorname{Path}^{bg}_\delta (M)$. Clearly, 
$d_\infty (\sigma, \gamma) \leq {\epsilon}$.  
It follows that the space of broken geodesics in $M$ is dense in the space of all paths in $M$, with the compact topology. 

Similarly, we can define the subspace of loops $\Omega^{bg}_{\delta}(M,n)\subset \operatorname{Path}^{bg}_\delta(M,n)$ of all broken geodesics with $n$ or less marked points, starting and ending at $\gamma (t_0) = \gamma (t_k) = x_0, k\leq n$. By the exact same arguments as above, $\displaystyle \Omega^{bg}(M)=\bigcup_{n\geq 1}\Omega^{bg}_{\delta}(M,n)$
is dense in $\Omega M$.
\end{proof}

\begin{cor} If $\gamma $ is any path in a compact Riemannian manifold $M$, then $\gamma$ is homotopic,  relative to its endpoints, to a broken geodesic.
\end{cor}

\begin{proof}
Let $\gamma $ be a path in $M$. By the proof Proposition \ref{density}, there exists a broken geodesic $\sigma$ with the same endpoints satisfying $d_\infty(\gamma, \sigma\big) < \operatorname{conv} (M)$.
By Lemma \ref{homotopicmaps}.
$\gamma$ and $\sigma$ are homotopic relative  to the endpoints.
 \end{proof}

\begin{rem}\rm 
Milnor (\cite{Milnor2}, Chapter 16) considers similar spaces of closed geodesics. These are loops $\gamma$ based at $x_0$, with fixed subdivisions $t_0=0< t_1< \cdots < t_n=1$ such that every restriction $\gamma_{|[t_i,t_{i+1}]}$ is a geodesic. This space is written $\Omega (t_0,\ldots, t_n)$. The union of all these subspaces, over all subdivisions, forms our $\Omega^{bg} (M)$.
\end{rem}

\subsection{Topology of uniform broken geodesics}
Let $(M,g)$ be a Riemannian manifold and $B$ a convex geodesic ball.
For each $(p,q)\in B\times B$ let $\gamma_{pq}:[0,1]\to M$ be the geodesic segment from $p$ to $q$ (i.e. the unique minimizing constant speed 
geodesic from $p$ to $q$).
Then the assignment
\begin{equation}\label{mapGamma}
\Gamma: B\times B \longrightarrow C^0([0,1],M),\qquad (p,q)\longmapsto \gamma_{pq}
\end{equation}
is continuous, where the path space has the compact-open topology, also known as the $C^0$ topology, or equivalently the topology of the sup metric \eqref{dinfinity}. 

Consider again the following subspace of $M^{n+1}$ 
\begin{equation}\label{ddelta}
    D_{\delta}(M,n)=\{(y_0,y_1,\ldots, y_{n})\in M^{n+1} \mid d(y_i, y_{i+1})\leq \delta,\  0\leq i< n\}.
\end{equation}
These are marked paths as in \eqref{markedpaths} which can start anywhere.
As in \eqref{imageonto}, 
there is a map 
$\displaystyle \Psi_n:D_{\delta}(M,n) \longrightarrow \operatorname{Path}(M) $
sending each consecutive pair $(y_i,y_{i+1})$ to the (unique) unit-speed minimal geodesic segment
$\left[0,d_M(y_i,y_{i+1})\right]\rightarrow M$, sending $0\mapsto y_i, d_M(y_i,y_{i+1})\mapsto y_{i+1}$. Concatenating these segments as in \eqref{productpath} and then dividing by the total length $\sum_{0\leq i<n} d_M(y_i,y_{i+1})$ yields a uniform broken geodesic $\gamma_{\bf m}: [0,1]\rightarrow M$ with marking $\bf m$.
When the total length is zero, we take \(\gamma_{\mathbf m}\) to be the constant path.

\begin{lem}\label{continuous} 
The map
$$\displaystyle \Psi_n:D_{\delta}(M,n) \longrightarrow\operatorname{Path}^{ubg}(M)\ \ ,\ \ m=(y_0,y_1,\ldots, y_n)\mapsto \gamma_{\bf m}$$ is continuous in the compact open topology for the path space.
\end{lem}

\begin{proof}
Fix
${ m}=(y_0,\ldots,y_n)\in D_\delta(M,n)$ and for each $i=0,\ldots,n-1$, set
$\ell_i({ m})=d_M(y_i,y_{i+1})$.
Write $\sigma_i({ m}):
[0,\ell_i({\bf m})]\longrightarrow M$
be the unique minimizing geodesic from $y_i$ to $y_{i+1}$, parametrized by
arclength. 
Since $
d_M(y_i,y_{i+1})\leq \delta<\operatorname{conv}(M),$
there exists, for each $i$, a geodesically convex open ball $B_i$ containing
both $y_i$ and $y_{i+1}$. After shrinking neighborhoods of the endpoints if
necessary, all nearby pairs
$(y_i',y_{i+1}')$
remain in $B_i$, and in that neighborhood of ${\bf m}$, the map
${ m}'\longmapsto \sigma_i({ m}')$ is continuous. Since concatenation of Moore paths is continuous, the Moore concatenation $\gamma^M_{ m}$ of all these segments is continuous. Its Moore length is $L(\gamma^M_{ m})$.
If $L(\gamma^M_{ m})=0$, the path is constant, otherwise the associated uniformly parametrized broken
geodesic is given by
$$
\gamma_{ m}(t)
=
\gamma^M_{ m}\bigl(L({ m})t\bigr),
\qquad 0\leq t\leq 1.
$$
Consider the adjoint map
\[
D_\delta(M,n)\times[0,1]\longrightarrow M,
\qquad
\tilde\gamma ({ m},t)
=
\gamma_{ m}\bigl(t\bigr)
\]
It is continuous because concatenation and the evaluation map are continuous functions.
Since $[0,1]$ is compact, the exponential law applies for the topological product and the compact-open topology, and implies that the adjoint
${ m}
\longmapsto 
\gamma_{ m}$
is continuous.
\end{proof}

We define the space $$\operatorname{Path}^{ubg}_\delta(M,n):=\Psi_n(D_\delta(M,n))$$ 
In words, this is the space of uniform broken geodesics which have a representation with at most $n$ marked points whose consecutive distances are at most $\delta$. 
Clearly $\operatorname{Path}^{ubg}_\delta(M,0)=M$, and $\operatorname{Path}^{ubg}_\delta(M,1)$ is a closed disk bundle over $M$.

\begin{cor}\label{compact}
Let $M$ be a compact Riemannian manifold. Then both $\operatorname{Path}^{ubg}_\delta(M,n)$ and $\Omega^{ubg}_\delta (M,n)$ are compact subspaces
of $\operatorname{Path}(M)$ and $\Omega M$ respectively.
\end{cor}

\begin{proof}  
The space $D_{\delta}(M,n)$ is a compact subspace of $M^{n+1}$. Its image under $\Psi_n$ in $\operatorname{Path}(M) $ is precisely
$\operatorname{Path}^{ubg}_\delta(M,n)$. Since the image of a compact set by a continuous map is compact, it follows that  its image, which is $\operatorname{Path}^{ubg}_\delta(M,n)$ is compact. Similalry, the image of the subspace of $D_\delta(M,n)$ of all tuples $(x_0,\ldots, x_n)$ with $x_n=x_0$ is $\Omega_\delta^{ubg}(M,n)$ and that too is compact.
\end{proof}

\section{The $H$-space structure on $H_{\infty}$}\label{hspace}

We have defined in \S\ref{marked} the $H$-space model
$H_\infty=\operatorname*{colim}_{n}H_n$
with the ordinary colimit topology, together with compatible continuous maps
\[
\mu_{n,m}:H_n\times H_m\longrightarrow H_{n+m}.
\]
These maps define a set-theoretic associative multiplication $\mu$ on $H_\infty$.
From the continuity of all the maps \(\mu_{n,m}\), one automatically obtains a continuous map
\[
\operatorname*{colim}_{n,m}(H_n\times H_m)
\longrightarrow H_\infty.
\]
However, in the category \(\mathbf{Top}\), one cannot in general identify the lef hand colimit
with
$\displaystyle
\left(\operatorname*{colim}_nH_n\right)
\times
\left(\operatorname*{colim}_mH_m\right).$
This is however true in the category \(k\mathbf{Top}\) as we now explain. 

For convenience, we briefly summarize the relevant features of the category $kTop$ of compactly generated spaces. A standard reference for the subject is \cite{strickland} to which we refer for the definition of $kTop$. 
For the following properties, we refer to \cite{ghazel,strickland}.

\begin{enumerate}[(i)]
\item $kTop$ is a coreflective subcategory of Top; this means that it is a full subcategory of Top and the inclusion functor $J:kTop \to Top$ has a right adjoint denoted as $k$ and is called a coreflector.
\item $kTop$ contains  the locally compact Hausdorff spaces and CW-complexes. All metric spaces are in $kTop$, in particular, manifolds and their loop spaces are CG.
\item All sequential spaces are in $kTop$. In particular,   first countable spaces  and metric spaces are CG. Moore loop spaces and path spaces of metric spaces are  CG.
\item Just as $Top$, the category $kTop$ is complete and cocomplete. Therefore it has products, pullbacks, coproducts, pushouts, equalizers and coequalizers.
\item Given any functor defined $F:I \to kTop $ with $I$ small. Its colimit is precisely the colimit of the composite $JF$. 
 As an example, let $X$ and $Y$ be two CG spaces. Let $X\times Y$  and $X\times_{k}Y$ denote their products $Top$ and $kTop$ respectively. Then
$$X\times_{k}Y \cong  k(X\times Y).$$ 
\item The category $kTop$ is cartesian closed. 
As a very important consequence, given any CG space $X$, the functor 
$X\times_{k}.: kTop \longrightarrow kTop$
commutes with colimits.\\
\end{enumerate}

The relevance of $kTop$ to this work is now presented.
The order relation on $\N$ induces a category structure on $\N$ in the standard way. For $n\in \N$, let $H_n$ be as defined in (\ref{Hn}) and  let   $$H: \N  \lra  kTop $$   be the functor  given by $H(n)= H_n$  and let $H_{\infty}\cong colim H.$  
Define a functor 
\begin{equation}
 F: \N \times \N \lra  kTop 
\end{equation}
  by $F((n,m))= H_n \times_k H_m$. 
The category $kTop$ being cartesian closed, for a fixed $n\in \N$, the partial functor 
 $F(n,.): \N \lra kTop$  has a colimit and $$\operatorname{colim} F(n,.)\cong H_n \times_k H_{\infty}.$$ 
Define a functor $\hat{F}:\N \lra kTop$ by $\hat{F}(n)=H_n \times_k H_{\infty}$. Again, $kTop$ is cartesian closed, therefore  $\hat{F}$ has a colimit and $\operatorname{colim} \hat{F}\cong H_{\infty} \times H_{\infty}$. By the Fubini's theorem \cite[Proposition, page 230]{Mac}, the functor $F$ itself has a colimit and $\operatorname{colim} F\cong H_{\infty} \times_k H_{\infty}$.  By \cite[Proposition 2.17.]{strickland}, there are continuous maps $$\mu_{n,m}:H_n \times_k H_m \lra H_{n+m}$$ rendering the diagram (\ref{Hmap}) commutative. The composite maps
$$\xymatrix{H_n \times_k H_m \ar[rr]^{\mu_{n,m}}&& H_{n+m} \ar[r]&  H_{\infty}} $$ 
define a cone $F \Rightarrow H_{\infty}$. This cone induces a map $\mu: H_{\infty} \times_k H_{\infty} \lra H_{\infty}$ which makes \(H_\infty\) a monoid object in \(k\mathbf{Top}\).


\section{The Holonomy of the Evaluation Map}\label{holonomyappendix}

This section is needed for the proof of Theorem \ref{XtoBG}, and so all loops and paths are ordinary (i.e. non-Moore). Let $X$ be a connected simplicial complex with a fixed $0$-cell $x_0$ (the basepoint). Let $M$ be a based space, $\bmap$ and $\map$ be the based and free mapping space functors, respectively. Evaluating a map $f$ at the basepoint $ev(f)=f(x_0)$ gives a fibration sequence
$$\xymatrix{\bmap (X,M)\ar[r]& \map (X,M)\ar[r]^-{ev}&M}$$
with holonomy \eqref{holonomy}
\begin{equation}\label{holonomyeval}
\Omega M\times\bmap (X,M)\lrar\bmap (X,M)
\end{equation}
Since $x_0$ is non-degenerate, there exists a small open neighborhood $V$ of it such that  $V$ is homeomorphic to the open cone of the boundary $\partial V$ of $V$. The neighborhood $V$ can be chosen to be a disk if $X$ is a surface, in which case  $\partial V$ is a circle. For general $X$, $\partial V$ may not be a circle.  Collapse that cone to a segment, then identify the endpoints of this segment. Doing this gives a coaction map $c: X\rightarrow S^1\vee X$. It turns out that holonomy for \eqref{holonomyeval} is simply precomposing maps with the coaction.

\begin{prop}\label{holeval}
The map
\begin{eqnarray*}
  \bmap (X\vee S^1,M) = \Omega M\times\bmap (X,M)\lrar\bmap (X,M)
\end{eqnarray*} 
which sends $(\gamma, f)\longmapsto (\gamma \vee f)\circ c$
is a choice of a holonomy for the evaluation fibration \eqref{holonomyeval} .
\end{prop}

\begin{proof}
Consider a general fibration $\pi : E \lrar B$ over a pointed path-connected base space $B$ with base point $b_0$, $F$ be the fiber of $\pi$ over $b_0$, and define 
  \begin{equation}\label{eq1}
  \tilde{E}=\{\gamma \in Map(I,E)\mid \gamma(0) \text{ and } \gamma(1) \text{ are both in } F\}. 
  \end{equation}
The projection $\pi$ induces a map $\tilde{\pi }:\tilde{E} \rightarrow \Omega B$.
 Let 
$ev_0, ev_1: \tilde{E} \rightarrow F$
be the evaluation map at $0$ and $1$ respectively. By the homotopy lifting property, the map \begin{equation}\label{eq20}
   (\tilde{\pi},ev_0): \tilde{E} \lrar \Omega B \times F  
  \end{equation} 
has a right inverse $
L: \Omega B \times F \mapsto \tilde{E}$
which is unique up to homotopy. The holonomy $h$ of the fibration $\pi$ can be regarded to be the composite map
$\xymatrix{\Omega B\times F \ar[r]^{\ \ \ L}&  \tilde{E}\ar[r]^{ev_1}& F}$.

Back to the space $X$, we know that $I\times X$ is a CW-complex containing $\{0\}\times X\cup I \times \{x_0\}$ as a subcomplex and a strong deformation retract. Construct the map
\[
k:I\times X\to S^1\vee X
\]
by retracting \(I\times X\) onto \(\{0\}\times X\cup I\times \{x_0\}\) and then identifying
\((0,x_0)\sim (1,x_0)\).
Write again
\begin{eqnarray*}
\widetilde{\map}(X,M)
&=&
\{\gamma \in \map (I, \map (X,M)) \mid \gamma(0) \text{ and } \gamma(1) \text{ are in } Map^*(X,M)\}
\end{eqnarray*}
 Precomposition with \(k\),
after identifying
$\map (I, \map (X,M))$ with 
$\map (I\times X,M)$, yields a map
\[
\bmap (S^1\vee X,M)\to \widetilde{\map}(X,M) 
\]
Identifying the source with
$
\Omega M\times \bmap (X,M)$ yields the lift map
\[
L: \Omega M\times \bmap (X,M)\to \widetilde{\map}(X,M) 
\]
which is  a right inverse of the map
$(\tilde{\pi}, ev_0):\widetilde{Map}(X,M) \rightarrow \Omega M \times Map^*(X, M)$
as in \eqref{eq20}. 
The composition of $L$ with the induced endpoint map $ev_1$is precisely
\[
(\gamma,f)\longmapsto (\gamma\vee f)\circ c
\]
which is precisely our claim. This is our lift of the loop \(\gamma\) in \(M\) starting at \(f\), and its endpoint is  \((\gamma\vee f)\circ c\). Therefore the map
\[
\Omega M\times \map^*(X,M)\to \map^*(X,M),
\qquad
(\gamma,f)\mapsto (\gamma\vee f)\circ c,
\]
is a choice of holonomy for the evaluation fibration.
\end{proof}

The important deduction we need for the proof of Theorem \ref{XtoBG} is that the restriction of holonomy to the one-cells of the one-skeleton is simply conjugation, and this is the case since the coaction \(c\) is the usual ``pull a loop out of the wedge point and glue it back'' construction, thus induces conjugation on each circle summand. This is recorded in the Lemma below.

\begin{lem}\label{one-skeleton-conjugation}
For every \(\gamma\in \Omega M\), $f\in\bmap (X,M)$, the restriction of the holonomy image of $f$, 
$
(\gamma\vee f)\circ c:X\to M
$
(notation as in Proposition \ref{holeval})
to the one-skeleton \(X^{(1)}\) corresponds to the \(n\)-tuple
\[
(f_1^\gamma,\dots,f_n^\gamma),
\qquad
f_i^\gamma:=\gamma^{-1} * f_i * \gamma.
\]
where
$f_i=f|_{S^1_i}\in \Omega M$ is the restriction of $f$ to the $i$-th wedge summand of \(X^{(1)}=\bigvee_{i=1}^n S^1\).
\end{lem}

\begin{proof}
By construction, the coaction
$c:X\to X\vee S^1$
is obtained by collapsing a small cone neighborhood of the basepoint \(x_0\) to a segment and then identifying the two endpoints of that segment. When restricted to the one-skeleton \(X^{(1)}\), this operation inserts the extra circle factor \(S^1\) at the wedge point of each summand \(S^1_i\). Thus the restriction
\[
c|_{X^{(1)}}:X^{(1)}\longrightarrow X^{(1)}\vee S^1
\]
acts on each circle \(S^1_i\) by first traversing the inverse of the added loop \(S^1\), then   \(S^1_i\) , and then the added loop. Equivalently, on the \(i\)-th summand it is the standard coaction representing conjugation.
\end{proof}

\bibliographystyle{amsalpha}

\begin{thebibliography}{99}

\bibitem{Abbas} A.~Bahri, {\em  Recent progress on conformal geometry}, Imperial College Press, Advanced Texts in Math. Vol.1 (2007).


\bibitem{bc} A~Bahri, F.~Cohen, \textit{On small geodesics and free loop spaces}, ArXiv:0806.0637v1.

\bibitem{bauer} U. Bauer, M. Kerber, F. Roll, A. Rolle, \textit{A unified view on the functorial nerve theorem and its variations}, 
Expo. Math. {\bf 41} (2023), no. 4, Paper No. 125503, 52 pp.

\bibitem{berger} M. Berger, \textit{A panoramic view of Riemannian geometry}, Springer-Verlag, Ed 2 (2007).

\bibitem{bott}
R. Bott and H. Samelson, \emph{On the Pontryagin product in spaces of paths}, Comment. Math. Helv. {\bf 27} (1953), 320–337.

\bibitem{burge} J. B\"urge, \textit{Neural Geodesic Flows}, Master’s Thesis, ETH Zurich (2025). 






\bibitem{docarmo} M. Do Carmo, \textit{Riemannian Geometry}, Birkhäuser, Boston, 1992.





\bibitem{pawel} P. Gajer, \textit{The intersection Dold-Thom theorem}, Topology {\bf 35}, Issue 4 (1996), 939--967.


\bibitem{ganea} T. Ganea, \textit{Induced fibrations and
  cofibrations}, Trans. Amer. Math. Soc. {\bf 127} (1967), 442--459.
  
\bibitem{ghazel} M. Ghazel, \textit{Kan extendable subcategories and fibrewise topology}, Advances in Pure and Applied Mathematics {\bf 15} (2024), 21--83. 
  
\bibitem{gk} M. Ghazel, S. Kallel, \textit{Thin loop spaces}, Eur. J. Math. 6 (2020), no. 3, 864--887. 

\bibitem{grusal} K. Gruher, P. Salvatore, \textit{Generalized string topology operations}, Proceedings of LMS {\bf 96} (1), (2008), 78--106.


\bibitem{hess} K. Hess, A. Tonks, \textit{The loop group and the cobar construction}, 
Proceedings of the American Mathematical Society
{\bf 138}, No. 5 (2010), 1861--1876.

\bibitem{hilton} P. Hilton, \textit{Homotopy theory and duality},
  Gordon and Breach Science Publishers 1965.


\bibitem{hu} S. T. Hu, \textit{Introduction To General Topology}, Holden-Day series in Mathematics (1966).





\bibitem{James} I.M. James, \textit{Reduced Product Spaces},   Annals of Mathematics, 2nd Ser. {\bf 62}, No.1 (1955), 170--197. 

\bibitem{kss} J. Klein, C. Schochet, and S. Smith, \textit{Continuous trace C*-algebras, gauge groups and rational homotopy}, Journal of Topology and Analysis {\bf 1}(3) (2009), 261--288.

\bibitem{lellig} J. Lillig, \textit{A union theorem for cofibrations}, Arch. Math {\bf 24} (1973), 410–415.

\bibitem{ali} A. Maalaoui, V. Martino,
 \textit{The topology of a subspace of the Legendrian
curves on a closed contact 3-manifold}, Advanced Nonlinear Studies, {\bf  14}:2 (2014) 393-426.

\bibitem{Mac} Mac Lane, S.: Categories for the working mathematician, 2nd ed. Grad. Texts in Math. {\bf 5}, Springer (1998).

\bibitem{may} P. May, \textit{Classifying Spaces and Fibrations}, Memoirs of the American Mathematical Society {\bf 155} (1975).

\bibitem{Milnor} J.~Milnor, {\em Construction of universal bundles, I}, Ann. of Math.,
(2)\textbf{63}(1956), 272--284.

\bibitem{Milnor2} J.~Milnor, {\em Morse Theory}, Princeton U. Press. 1963.

\bibitem{neisendorfer} J. Neisendorfer, \textit{Algebraic Methods in Unstable Homotopy Theory}, Cambridge University Press, 2010. 


\bibitem{oancea} A. Oancea, \textit{Morse theory, closed geodesics, and the homology of free loop spaces}, IRMA Lect. Math. Theor. Phys. {\bf 24}, European Mathematical Society (EMS), Z\"urich (2015), 67–109.


\bibitem{piacenza} R.J. Piacenza, \textit{Homotopy theory of diagrams and CW-complexes over a category}, Can. J. Math.Vol. {\bf 43} (4) (1991), 814-824.

\bibitem{schon}  R.~Sch{\"o}n, \textit{Fibrations over a CWh-Base}, Proc. Amer. Math. Soc. {\bf 62} No. 1 (1977),
165-166.

\bibitem{smale} S. Smale, \textit{Regular curves on Riemannian manifolds}, Trans. Amer. Math.
Soc. {\bf 87}, 1958, 492--512.

\bibitem{spanier} 
{E. H. Spanier}
{Algebraic Topology},
{McGraw-Hill},
{New York} (1966).
\bibitem{stasheff} J. Stasheff, \textit{Parallel transport and classification of fibrations}, Algebraic and Geometrical Methods in Topology, Springer Lecture Notes in Math. {\bf 428} (1974), 1--17. 
\bibitem{stasheff2} J. Stasheff, \textit{A classification theorem for fibre spaces}, Topology {\bf 2} (1963), 239–246.
\bibitem{strickland}  N.P. Strickland, \textit{The category of CGWH spaces}, https://neil-strickland.staff.shef.ac.uk/courses/
homotopy/cgwh.pdf

\end{thebibliography}

\end{document}